\documentclass[a4paper,reqno]{amsart}
\usepackage{amsmath}
\usepackage{amssymb}
\usepackage{amsthm}
\usepackage{latexsym}
\usepackage{amscd}
\usepackage{xypic}
\xyoption{curve}
\usepackage{ifthen} 
\usepackage{hyperref} 
\usepackage{graphicx}
\usepackage{enumerate}
\usepackage{enumitem}
\usepackage{rotating}
\usepackage{xcolor}
\usepackage{mathtools}
\usepackage{todonotes}
\usepackage{color,soul}

\usepackage{float}
\restylefloat{table}

\numberwithin{equation}{section}

\def\cocoa{{\hbox{\rm C\kern-.13em o\kern-.07em C\kern-.13em o\kern-.15em A}}}

\newtheorem{theorem}{Theorem}[section]
\newtheorem{question}[theorem]{Question}
\newtheorem{lemma}[theorem]{Lemma}
\newtheorem{proposition}[theorem]{Proposition}
\newtheorem{corollary}[theorem]{Corollary}

\theoremstyle{definition}
\newtheorem{remark}[theorem]{Remark}

\newtheorem{example}[theorem]{Example}

\newtheorem{construction}[theorem]{Construction}

\newcommand {\sHom}{\mathcal{H}\kern -0.25ex{\mathit om}}
\newcommand {\sExt}{\mathcal{E}\kern -0.25ex{\mathit xt}}
\newcommand {\sTor}{\mathcal{T}\kern -0.25ex{\mathit or}}
\newcommand {\im}{\mathrm{im}}
\newcommand {\rk}{\mathrm{rk}}

\newcommand {\Ext}{\mathrm{Ext}}
\newcommand {\Hom}{\mathrm{Hom}}
\newcommand {\Hilb}{\mathcal{H}\kern -0.25ex{\mathit ilb\/}}

\newcommand {\cK}{\mathcal{K}}
\newcommand {\cA}{\mathcal{A}}
\newcommand {\cB}{\mathcal{B}}

\newcommand{\cC}{{\mathcal C}}
\newcommand{\cS}{{\mathcal S}}
\newcommand{\cE}{{\mathcal E}}
\newcommand{\cF}{{\mathcal F}}

\newcommand{\cN}{{\mathcal N}}
\newcommand{\cO}{{\mathcal O}}
\newcommand{\cG}{{\mathcal G}}

\newcommand{\cI}{{\mathcal I}}

\newcommand {\bZ}{\mathbb{Z}}

\newcommand {\bC}{\mathbb{C}}
\newcommand {\bP}{\mathbb{P}}

\newcommand {\bF}{\mathbb{F}}

\newcommand{\Pic}{\operatorname{Pic}}

\makeatletter
\@namedef{subjclassname@2020}{\textup{2020} Mathematics Subject Classification}
\makeatother

\def\p#1{{\bP^{#1}}}

\def\mapright#1{\mathbin{\smash{\mathop{\longrightarrow}
\limits^{#1}}}}

\title[Instanton bundles on $\p1\times\bF_1$]{Instanton bundles on $\p1\times\bF_1$}

\thanks{The first and second authors are a member of GNSAGA group of INdAM and are supported by the framework of the MIUR grant Dipartimenti di Eccellenza 2018-2022 (E11G18000350001). The third author is supported by the grant MAESTRO NCN - UMO-2019/34/A/ST1/00263 - Research in Commutative Algebra and Representation Theory. The third author is also partially supported by Narodowe Centrum Nauki 2018/30/E/ST1/00530. The authors are very grateful for the hospitality and support of the IMPAN, Krakow during part of the preparation of this paper. }



\author[V. Antonelli, G. Casnati, O. Genc]{V. Antonelli, G. Casnati, O. Genc}

\begin{document}

\maketitle

\begin{abstract}
In this paper we deal with a particular class of rank two vector bundles (\emph{instanton} bundles) on the Fano threefold of index one $F:=\mathbb{F}_1 \times \mathbb{P}^1$. We show that every instanton bundle on $F$ can be described as the cohomology of a monad whose terms are free sheaves. Furthermore we prove the existence of instanton bundles for any admissible second Chern class and we construct a nice component of the moduli space where they sit. Finally we show that minimal instanton bundles (i.e. with the least possible degree of the second Chern class) are aCM and we describe their moduli space.
\end{abstract}

\section{Introduction}
The study of vector bundles on a variety $X$ is a fruitful field of research in algebraic geometry. In particular  the study of vector bundles $\cE$ with some prescribed cohomology groups gives interesting information on the variety itself. 

For example, when $X$ is endowed with a (very) ample polarization $\cO_X(h)$, then it is natural to deal with bundles defined by some particular vanishings. E.g., on the one hand it is interesting to understand when the bundle $\cE$ is {\sl aCM}, i.e. $h^i\big(X,\cE(th)\big)=0$ for each $t\in\bZ$ and $0< i<\dim(X)$. On the other hand, one can also try to classify bundles with very few vanishings.

Examples of vector bundles of the latter type are provided by instanton bundles on a {\sl Fano threefold} $X$, i.e. a smooth threefold over the complex field $\bC$ whose anticanonical line bundle $\omega_X^{-1}$ is ample: see \cite{I--P} for the classification of such threefolds. In particular, the {\sl index of $X$} is the greatest $i_X\in\bZ$ such that $\omega_X\cong\cO_X(-i_Xh)$ for an ample $\cO_X(h) \in \Pic(X)$. It is well--known that $1\le i_X\le 4$ and that such an $\cO_X(h)$ is uniquely determined: it is called the {\sl fundamental line bundle} of $X$. 

The easiest Fano threefold is $\p3$: an instanton bundle $\cE$ on $\p3$ is a rank $2$ bundle such that $h^0\big(\p3,\cE\big)=h^1\big(\p3,\cE(-2)\big)=0$ and $c_1(\cE)=0$. Instanton bundles on $\p3$ were first introduced in the seminal paper \cite{A--D--H--M} and widely studied from different viewpoints since the discovery of their connection, through the Atiyah--Penrose--Ward transformation, with the solutions of the Yang--Mills equations (see \cite{A--W}).
 
Because of its importance, the notion of instanton bundle has been extended in several papers to different classes of  varieties. Without any claim of completeness we refer the interested reader to \cite{Fa, Kuz} for instanton bundles on a Fano threefold $X$. If $\varrho_X=1$, then the condition $h^0\big(X,\cE\big)=0$ is equivalent to the $\mu$--stability of $\cE$ with respect to $\cO_X(h)$ (see \cite{H--L} for the definition and properties of $\mu$--(semi)stable bundles). When $\varrho_X\ge2$ this is no longer true. For this reason in \cite{C--C--G--M} instanton bundles are defined as rank $2$ bundles $\cE$ such that $h^0\big(X,\cE\big)=h^1\big(X,\cE(-q_Xh)\big)=0$, $c_1(\cE)=(2q_X-i_X)h$, where $q_X:=\left[\frac{i_X}2\right]$, and which are $\mu$--semistable with respect to $\cO_X(h)$. The study of such instanton bundles has been carried out on several Fano threefolds in \cite{M--M--PL, A--M,C--C--G--M, Cs--Ge}: see also \cite{A--C--G} for further generalizations. We collect in Section \ref{sGeneral} the results on instanton bundles that we need in the paper.

In this paper we deal with instanton bundles on $F:=\p1\times\bF_1$, where $\bF_1\subseteq\p8$ is the del Pezzo surface of degree $8$ obtained by blowing up a single point in $\p2$. In order to better explain the results proved in the paper we fix some notation. The threefold $F$ is endowed with two natural projections $p\colon F\to \bF_1$ and $\pi\colon F\to\p1$. If $\sigma\colon\bF_1\to\p2$ is the blow up morphism, then $\Pic(\bF_1)$ is freely generated by the classes $\ell$ of a the pull--back via $\sigma$ of a general line in $\p2$ and $e$ of the exceptional divisor. The intersection theory on $\bF_1$ is given by $\ell^2=-e^2=1$ and $\ell e=0$. 

The group $A(F)$ decomposes as $\pi^*A(\p1)\otimes p^*A(\bF_1)$, hence it is free. Let $\xi$ be the pull--back of a point via $\pi$. By abuse of notation, we will denote by $\ell$ and $e$ also the pull--backs of the generators of $\Pic(\bF_1)$ with the same name via $p$. The dual $\cO_F(h)$ of
$$
\omega_F\cong p^*\omega_{\bF_1}\otimes\pi^*\omega_{\p1}\cong\cO_F(-3\ell+e-2\xi).
$$
is very ample and not divisible in $\Pic(F)$, hence $F$ is a Fano threefold and $i_F=1$. Thanks to the decomposition of $A(F)$, we deduce that $A^2(F)$ has generators $\ell^2$, $\ell \xi$, $e\xi$ where $\ell^2=-e^2$ is the class of a fibre of $p$ and $\ell e=\xi^2=0$. Finally, the same decomposition also yields that $A^3(F)$ is generated by $\ell^2\xi$. See Section \ref{sFano} for some other results on $F$ used in the paper.

Our first result proved in Section \ref{sMonad} is the description of each instanton bundle on $F$ as the cohomology of a monad. More precisely, for every choice of integers $\alpha,\beta,\gamma,\delta,\epsilon$ such that
\begin{gather}
\label{BoundInstanton1}
 \alpha\ge3,\qquad \gamma\ge2,\qquad \alpha+\gamma\ge6,\qquad \alpha-\beta\ge2\\ 
\label{BoundInstanton2}
\epsilon\ge2-\beta-\gamma,\qquad \delta\ge1-\beta.
\end{gather}
we define
\begin{gather*}
\cC^{-1}:=\cO_{F}(-2\ell+e-\xi)^{\oplus\alpha+\gamma-6}\oplus\cO_{F}(-\ell-\xi)^{\oplus\epsilon},\\
\begin{align*}
\cC^0:=\cO_{F}(-\ell-\xi)^{\oplus\beta+\gamma+\epsilon-2}&\oplus\cO_{F}(-\ell+e-\xi)^{\oplus\alpha-\beta+\gamma-4}\oplus\\
&\oplus\cO_F(-2\ell+e)^{\oplus\alpha-3}\oplus\cO_F(-\ell)^{\oplus\delta},
\end{align*}\\
\cC^1:=\cO_{F}(-\xi)^{\oplus\gamma-2}\oplus\cO_{F}(-\ell)^{\oplus\beta+\delta-1}\oplus\cO_{F}(-\ell+e)^{\oplus\alpha-\beta-2}.
\end{gather*}
In Section \ref{sMonad} we prove that if $\cE$ is an instanton bundle with $c_2(\cE)=\alpha\ell\xi-\beta e\xi+\gamma\ell^2$, then inequalities \eqref{BoundInstanton1} necessarily hold  (see Remark \ref{rInequalities}). 

The main result of the section is as follows.

\begin{theorem}
\label{tMonad}
 Each  instanton bundle $\cE$ with $c_2(\cE)=\alpha\ell\xi-\beta e\xi+\gamma\ell^2$ is the cohomology of a monad $\cC^\bullet$ of the form
\begin{equation}
\label{Monad}
0\longrightarrow \cC^{-1}\longrightarrow \cC^0\longrightarrow\cC^1\longrightarrow0
\end{equation}
where $\delta:=h^1\big({F},\cE(-e)\big)$ and $\epsilon:=h^1\big({F},\cE(-e-\xi)\big)$.

Conversely, if the cohomology $\cE$ of the monad $\cC^\bullet$ is a $\mu$--semistable bundle, then $\cE$ is an instanton bundle with $c_2(\cE)=\alpha\ell\xi-\beta e\xi+\gamma\ell^2$ on $F$ such that
\begin{enumerate}
\item $h^1\big({F},\cE(-e)\big)\le \delta$;
\item $h^1\big({F},\cE(-e-\xi)\big)\le \epsilon$;
\item $h^1\big(F,\cE(-D)\big)=0$ for each effective integral divisor $D\not\in\vert e\vert$.
\end{enumerate}
\end{theorem}

An almost immediate consequence of Theorem \ref{tMonad} is that an instanton bundle $\cE$ on $F$ (if any) satisfies $h^1\big(F,\cE(-D)\big)=0$ whenever $\vert D\vert\ne\emptyset$ contains smooth integral elements if and only if $h^1\big({F},\cE(-e)\big)=0$ (see  Corollary \ref{cBound}): following \cite{C--C--G--M} we call such bundles earnest.

 It is easy to construct examples of decomposable instanton bundles (see Example \ref{eDecomposable}). Thus, it is natural to ask if indecomposable instanton bundles actually exist on $F$.
 
We positively answer this question in Section \ref{sInstanton} showing that for $\alpha, \beta,\gamma\in \bZ$ satisfying inequalities \eqref{BoundInstanton1}, there exists a $\mu$--stable instanton bundle $\cE$ on $F$ with $c_2(\cE)=\alpha\ell\xi-\beta e\xi+\gamma \ell^2$ via the classical Serre correspondence (see Construction \ref{conInstanton}): in particular, monads $\cC^\bullet$ as in the previous statement actually exist. 

More precisely we prove the following result.

\begin{theorem}
\label{tConstruction}
For each choice of integers $\alpha,\beta,\gamma$ satisfying inequalities \eqref{BoundInstanton1} the bundle $\cE$ obtained via Construction \ref{conInstanton} is a $\mu$--stable instanton bundle with $c_2(\cE)=\alpha\ell\xi-\beta e\xi+\gamma\ell^2$, such that $\cE\otimes\cO_L\cong\cO_{\p1}\oplus\cO_{\p1}(-1)$ for the general line $L\subseteq F$. Moreover, the following assertions hold.
\begin{itemize}
\item If $\beta\ge1$, then $\cE$ is earnest.
\item If $\beta\le0$, then $\cE$  is non--earnest and 
$$
-\beta\le h^1\big(F,\cE(-e-\xi)\big)\le h^1\big(F,\cE(-e)\big)=1-\beta.
$$
\end{itemize}
\end{theorem}

Thanks to \cite[Theorem 1.8]{A--C--G} each indecomposable instanton bundle on a Fano threefold is simple. Let $\cS_F(2;-h,\alpha\ell\xi-\beta e\xi+\gamma \ell^2)$ be the moduli space of simple rank $2$ bundles $\cE$ on $F$ with $c_1(\cE)=-h$ and $c_2(\cE)=\alpha\ell\xi-\beta e\xi+\gamma \ell^2$ (see \cite{A--K}) and denote by
$$
\cS\cI_F(\alpha\ell\xi-\beta e\xi+\gamma \ell^2)\subseteq\cS_F(2;-h,\alpha\ell\xi-\beta e\xi+\gamma \ell^2)
$$
the locus of points corresponding to instanton bundles. The semicontinuity of $h^1\big(F,\cE\big)$ implies that such a locus is open, hence it can be considered as a moduli space for instanton bundles on $F$ with second Chern class as above. 

The instanton bundles obtained via Construction \ref{conInstanton}, being $\mu$--stable, are simple. In Proposition \ref{pModuli} we show that they correspond to smooth points of the same component inside $\cS\cI_F(\alpha\ell\xi-\beta e\xi+\gamma \ell^2)$. 

In Remark \ref{rBound} we recall that necessarily $c_2(\cE)h\ge14$ for each instanton bundle $\cE$ on $F$.  In Section \ref{sMinimal} we completely characterize instanton bundles for which the equality holds. In order to deal with such a characterization, notice that the projection $\p2\dashrightarrow \p1$ from the blown up point induces a morphism $\pi_0\colon {\bF_1}\to\p1$, hence there is a surjective morphism $\varphi:=id\times \pi_0\colon F\to Q:=\p1\times\p1$. If $q_i\colon Q\to\p1$ is the projection on the $i^{\mathrm{th}}$ factor,  we set $\cO_Q(\sigma_i):=q^*_i\cO_{\p1}(1)$. 

The main result of Section \ref{sMinimal}  is the following.

\begin{theorem}
\label{tMinimal}
There exist instanton bundles $\cE$ on $F$ such that $c_2(\cE)h=14$.

Every such an instanton bundle is aCM, $\mu$--stable, earnest, satisfies $\cE\otimes\cO_L\cong\cO_{\p1}\oplus\cO_{\p1}(-1)$ for the general line $L\subseteq F$ and $c_2(\cE)$ is either $4\ell\xi-2 e\xi+2 \ell^2$ or $3\ell\xi- e\xi+3 \ell^2$.
\begin{itemize}
\item $c_2(\cE)=4\ell\xi-2 e\xi+2 \ell^2$ if and only if  $\cE(\ell)=\varphi^*\cA$ where $\cA$ is $\mu$--stable with respect to $\cO_Q(\sigma_1+\sigma_2)$  such that $c_1(\cA)=-2\sigma_1-\sigma_2$ and $c_2(\cA)=2$. 
\item $c_2(\cE)=3\ell\xi-e\xi+3 \ell^2$ if and only if  $\cE(\xi)=p^*\cB$ where $\cB$ is $\mu$--stable with respect to  $\cO_{\bF_1}(2\ell-e)$  such that $c_1(\cB)=-3\ell+e$ and $c_2(\cB)=3$. 
\end{itemize}

Finally, both $\cS\cI_F(3\ell\xi-e\xi+3 \ell^2)$ and $\cS\cI_F(4\ell\xi-2e\xi+2 \ell^2)$ are smooth, irreducible and rational of dimension $1$.
\end{theorem}

We finally make some general comments on minimal instanton, weakly Ulrich, Ulrich and aCM bundles on Fano threefolds (see Section \ref{sMinimal} for the definitions and more details on the partial results that we obtain).

\subsection{Acknowledgements}
The authors express their thanks to the reviewers for their questions, remarks, suggestions which have considerably improved the whole exposition.

\section{Preliminary results}
\label{sGeneral}
Let $X$ be any smooth variety with canonical line bundle $\omega_X$. The Riemann--Roch formulas for a vector bundle $\cF$ on a smooth variety $X$ of dimensions $3$ and $2$ are respectively
\begin{align}
  \label{RRgeneral}
 & \begin{aligned}
    \chi(\cF)&=\rk(\cF)\chi(\cO_X)+{\frac16}(c_1(\cF)^3-3c_1(\cF)c_2(\cF)+3c_3(\cF))\\
    &-{\frac14}(\omega_Xc_1(\cF)^2-2\omega_Xc_2(\cF))+{\frac1{12}}(\omega_X^2c_1(\cF)+c_2(\Omega_{X}) c_1(\cF))
  \end{aligned}\\
  \label{RRsurface}
 &   \chi(\cF)=\rk(\cF)\chi(\cO_X)+{\frac12}c_1(\cF)^2-{\frac12}\omega_Xc_1(\cF)-c_2(\cF)
\end{align}
(see \cite[Theorem A.4.1]{Ha2}).

If $\cF$ and $\cG$ are coherent sheaves on $X$, then  the Serre duality holds
\begin{equation}
\label{Serre}
\Ext_X^i\big(\cG,\cF\otimes\omega_X\big)\cong \Ext_X^{\dim(X)-i}\big(\cF,\cG\big)^\vee
\end{equation}
(see \cite[Proposition 7.4]{Ha3}). 

Let $\cF$ be a vector bundle of rank $2$ on a smooth variety $X$ of dimension $n$ and let $s\in H^0\big(X,\cF\big)$. If we consider its zero locus
$(s)_0\subseteq X$ then we observe that it is either empty or its codimension is at most
$2$. In particular we can write $(s)_0=Z\cup S$ as the union of two components
where $Z$ has pure codimension $2$ (or it is empty) and $S$ is either empty or an effective divisor. Thus $\cF(-S)$ has a section vanishing
on $Z$ and we can consider its Koszul complex 
\begin{equation}
  \label{seqSerre}
  0\longrightarrow \cO_X(S)\longrightarrow \cF\longrightarrow \cI_{Z\vert X}(-S)\otimes\det(\cF)\longrightarrow 0.
\end{equation}
Sequence \ref{seqSerre} tensored by $\cO_Z$ yields $\cI_{Z\vert X}/\cI^2_{Z\vert X}\cong\cF^\vee(S)\otimes\cO_Z$, whence
\begin{equation}
\label{Normal}
\cN_{Z\vert X}\cong\cF(-S)\otimes\cO_Z.
\end{equation}
Notice that if $S=0$, then $Z$ is locally complete intersection inside $X$ and it has no embedded components.

The above construction can be reversed thanks to the Serre correspondence as follows.

\begin{theorem}
  \label{tSerre}
Let $X$ be a smooth variety of dimension $n\ge2$ and let $Z\subseteq X$ be a locally complete intersection subscheme of dimension $n-2$.  

If $\det(\cN_{Z\vert X})\cong\cO_Z\otimes\mathcal L$ for some $\mathcal L\in\Pic(X)$ such that $h^2\big(X,\mathcal L^\vee\big)=0$, then there exist vector bundles $\cF$ of rank $2$ on $X$ such that:
  \begin{enumerate}
  \item $\det(\cF)\cong\mathcal L$;
  \item $\cF$ has a section $s$ such that $Z$ coincides with the zero locus $(s)_0$ of $s$.
  \end{enumerate}
Moreover, the vector bundles $\cF$ as above are parameterized by a projective space of dimension $h^1\big(X,{\mathcal L}^\vee\big)$.
\end{theorem}
\begin{proof}
The vector bundles $\cF$ mentioned in the statement are parameterized by the points of the projectivized of $\Ext_X^1\big(\cI_{Z\vert X}\otimes\mathcal L,\cO_X\big)$ which has dimension $h^1\big(X,{\mathcal L}^\vee\big)$ as explained in the proof of \cite[Theorem 1.1]{Ar}: for further details see \cite{Ar}.
\end{proof}

A vector bundle $\cF$ on $X$ is {\sl simple} if $h^0\big(X,\cF\otimes\cF^\vee\big)=1$. In particular simple vector bundles are indecomposable.

\begin{lemma}
\label{lExt3}
If $\cF$ is a simple vector bundle on a smooth variety $X$ such that $h^0\big(X,\omega_X^{-1}\big)>0$ and $h^0\big(X,\omega_X^{\rk(\cF)}\big)=0$, then
$$
\Ext^{\dim(X)}_{X}\big(\cF,\cF\big)=0.
$$
\end{lemma}
\begin{proof}
See \cite[Lemma 2.1]{C--C--G--M}.
\end{proof}

If $\cO_X(H)$ is an ample line bundle, then we define the {\sl slope} of a vector bundle $\cF$ on $X$ with respect to $\cO_X(H)$ as $\mu(\cF):= c_1(\cF)H^{n-1}/\rk(\cF)$. The bundle $\cF$ is {\sl $\mu$--stable} (resp. {\sl $\mu$--semistable}) with respect to $\cO_X(H)$ if  $\mu(\mathcal G) < \mu(\cF)$ (resp. $\mu(\mathcal G) \le \mu(\cF)$) for each subsheaf $\mathcal G$ with $0<\rk(\mathcal G)<\rk(\cF)$. Every $\mu$--stable bundle $\cF$ is simple (see \cite[Corollary 1.2.8]{H--L}). The following criterion for $\mu$--(semi)stability of a vector bundle on $F$ will be used in what follows.

\begin{lemma}
\label{lHoppe}
Let $\cF$ be a rank $2$ vector bundle on a smooth variety $X$.

Then $\cF$ is $\mu$--stable (resp. $\mu$--semistable) if and only if
$$
h^0\big(X,\cF(-D)\big)=0,
$$
for each divisor $D$ such that $\mu(\cO_X(D))\ge\mu(\cF)$ (resp. $\mu(\cO_X(D))>\mu(\cF)$).
\end{lemma}
\begin{proof}
The statement is a particular case of \cite[Corollary 4]{J--M--P--S}.
\end{proof}

Finally, we focus on the case of a Fano threefold $X$ with $i_X=1$. Let $\cE$ be a rank $2$ vector bundle with $c_1(\cE)=-h$. Equality \eqref{Serre} and the isomorphism $\cE^\vee\cong\cE(h)$ yield
\begin{equation}
\label{Serre1}
h^i\big(X,\cE(D)\big)=h^{3-i}\big(X,\cE(-D)\big)
\end{equation}
for each line bundle $\cO_F(D)\in\Pic(X)$. In particular $h^0\big(X,\cE\big)=h^3\big(X,\cE\big)$, hence the following lemma is easy to prove.

\begin{lemma}
\label{lNatural}
Let $X$ be a Fano threefold with $i_X=1$.

A $\mu$--semistable bundle $\cE$ of rank $2$ on $X$ such that $c_1(\cE)=-h$ is an instanton bundle if and only if $h^i\big(X,\cE\big)=0$ for each $i$.
\end{lemma}
\begin{proof}
See \cite[Proposition 4.1]{C--C--G--M}.
\end{proof}

If $\cE$ is an indecomposable instanton bundle on $X$, then $\cE$ is simple (see \cite[Theorem 1.8]{A--C--G}). It follows that $h^0\big(X,\cE\otimes\cE^\vee\big)=1$ by definition and $h^3\big(X,\cE\otimes\cE^\vee\big)=0$ by Lemma \ref{lExt3}. Thus equality \eqref{RRgeneral} for $\cE\otimes\cE^\vee$ yields
\begin{equation}
\label{Ext12}
\dim\Ext^1_{X}\big(\cE,\cE\big)-\dim\Ext^2_{X}\big(\cE,\cE\big)=2c_2(\cE) h-\frac{\deg(X)}2-3,
\end{equation}
because $hc_2(\Omega_{X})={24}$ (see \cite[Exercise A.6.7]{Ha2}).

\section{The threefold $F$}
\label{sFano}
In this section we list some basic results on the  threefold $F:=\p1\times \bF_1$ where $\sigma\colon \bF_1\to\p2$ is the blow up at a fixed point. 

We now recall the notation and some facts from the Introduction about the morphisms $p\colon F\to\bF_1$, $\sigma\colon\bF_1\to\p2$, $\pi_0\colon {\bF_1}\to\p1$, $q_i\colon Q\to\p1$.

We will use the same symbols $\ell$ and $e$ for denoting both the classes of the pull--back of the general line via $\sigma$ and of the exceptional divisor of $\sigma$, and their pull--backs via $p$. 

There is an isomorphism ${\bF_1}\cong\bP(\cO_{\p1}\oplus\cO_{\p1}(-1))$ and $\pi_0$ coincides with the standard projective bundle morphism (see \cite[Example V.2.11.5]{Ha2}). It follows that the classes of a fibre of $\pi_0$ and of the unisecant of minimal self--intersection $-1$ inside $A^1(\bF_1)\cong\Pic(\bF_1)$ are $\ell-e$ and $e$ respectively. In particular $\vert e\vert$ contains a single divisor $E\cong\p1$. 

Finally, the line bundles $\cO_Q(\sigma_i)$ are exactly the standard spinor bundles on $Q$. We have $\Omega_{Q}\cong\cO_Q(-2\sigma_1)\oplus\cO_Q(-2\sigma_2)$, $\cO_F(\xi)\cong\varphi^*\cO_Q(\sigma_1)$, $\cO_F(\ell-e)\cong\varphi^*\cO_Q(\sigma_2)$ and $F\cong\bP(\cO_{Q}\oplus\cO_Q(-\sigma_2))$. In particular, $\cO_F(e)\cong \cO_{\bP(\cO_{Q}\oplus\cO_Q(-\sigma_2))}(1)$, hence $\Omega_{F\vert Q}\cong\cO_F(-\ell-e)$, thanks to \cite[Exercise III.8.4 (b)]{Ha2}. Thus a Chern classes computation on the exact sequence
$$
0\longrightarrow \varphi^*\Omega_{Q}\longrightarrow \Omega_{F}\longrightarrow \Omega_{F\vert Q}\longrightarrow0,
$$
yields $c_2(\Omega_{F})=c_2(\Omega_{F}^\vee)=6\ell\xi-2e\xi+4\ell^2$.

Since $u\ell-ve=u(\ell-e)+(u-v)e$, it follows from \cite[Exercise III.8.4 (a)]{Ha2} that
\begin{equation}
\label{dimh^0}
h^0\big({\bF_1},\cO_{\bF_1}(u\ell-ve)\big)=\left\lbrace\begin{array}{ll} 
{{u+2}\choose2}-{{v+1}\choose2}\quad&\text{if $u,u-v\ge 0$,}\\
0\quad&\text{otherwise.}\\
\end{array}\right.
\end{equation}

The K\"unneth formula gives
\begin{equation}
\label{Kuenneth}
\begin{aligned}
h^i\big(F,\cO_F(a\ell-be+c\xi)\big)&=h^i\big({\bF_1},\cO_{\bF_1}(a\ell-be)\big)h^0\big(\p1,\cO_{\p1}(c)\big)+\\
&+h^{i-1}\big({\bF_1},\cO_{\bF_1}(a\ell-be)\big)h^1\big(\p1,\cO_{\p1}(c)\big).
\end{aligned}
\end{equation}
In particular, the linear system $\vert e\vert$ on $F$ contains a single divisor $E$ which is the pull--back of $E\subseteq\bF_1$. By abuse of notation we still denote it by $E$. Notice that $E\cong\p1\times\p1$.

\begin{remark}
\label{rNEF}
The line bundles $\cO_F(\ell)$, $\cO_F(\ell-e)$, $\cO_F(\xi)$ are globally generated, hence numerically effective, because they are the pull--back of globally generated line bundles on either ${\bF_1}$ or $\p1$. 
\end{remark}

On $F$ there are several interesting families of lines, conics and rational cubics (with respect to $\cO_F(h)\cong\omega_F^{-1}$) that will be used later on for constructing instanton bundles.

\begin{example}
\label{eLine}
The first important family of curves is the one of lines.

If $\alpha\ell\xi-\beta e\xi+\gamma\ell^2\in A^2(F)$ is the class of the curve $L\subseteq F$, then both $L(3\ell-e)$ and $L\xi$ are non--negative thanks to Remark \ref{rNEF}. In particular, if $L$ is a line, then the equality  $1=Lh=L(3\ell-e)+2L\xi$ implies
$$
1=L(3\ell-e)=3\alpha-\beta,\qquad 0=L\xi=\gamma.
$$
It follows that $L$ is the intersection of an element in $\vert\xi\vert$ with the pull--back of a divisor $\Delta\subseteq {\bF_1}$ such that $\Delta(3\ell-e)=1$, i.e. $\Delta$ is a line with respect to the embedding ${\bF_1}\subseteq\p8$ induced by $\cO_{\bF_1}(3\ell-e)$. The equality $p_a(\Delta)=0$ then forces $\Delta=E$. 

Thus the class of $L$ inside $A^2(F)$ is $e\xi$ and there is a Koszul--type resolution 
$$
0\longrightarrow\cO_{F}(-e-\xi)\longrightarrow\cO_{F}(-e)\oplus\cO_{F}(-\xi)\longrightarrow\cI_{L\vert F}\longrightarrow0,
$$
In particular $\cI_{L\vert F}/\cI_{L\vert F}^2\cong\cI_{L\vert F}\otimes\cO_L\cong\cO_{\p1}\oplus\cO_{\p1}(1)$. Thus $\cN_{L\vert F}\cong\cO_{\p1}\oplus\cO_{\p1}(-1)$.

Let $\Lambda$ be the Hilbert scheme of lines on $F$. The above construction shows that $\Lambda\cong\p1$, hence it is irreducible.
\end{example}

\begin{example}
\label{eA}
The line bundles $\cO_{F}(\ell)$ and $\cO_F(\xi)$ are globally generated  (see Remark \ref{rNEF}) and each element in $\vert \ell\vert$ is isomorphic to $\p1\times\p1$. Each general element in $\vert \xi\vert$ intersects such a divisor along a curve $A$ which is isomorphic to a general line in $\p2$ via $\sigma\circ p$: notice that $Ah=3$, hence the Hilbert polynomial of such a curve $A$ is $\chi(\cO_A(t))=3t+1$. The class of $A$ in $A^2(F)$ is $\ell\xi$ and $\cI_{A\vert F}$ fits into the Koszul--type resolution
$$
0\longrightarrow\cO_{F}(-\ell-\xi)\longrightarrow\cO_{F}(-\ell)\oplus\cO_{F}(-\xi)\longrightarrow\cI_{A\vert F}\longrightarrow0,
$$
hence $\cN_{A\vert F}\cong\cO_{\p1}\oplus\cO_{\p1}(1)$. Since $A$ is contained in the pull--back of a divisor on ${\bF_1}$ which does not intersect the exceptional divisor it follows that $A\cap E=\emptyset$, hence $A\not\subseteq E$.

We denote by $\Gamma_A$ the locus of curves as above inside the Hilbert scheme of curves on $F$ with Hilbert polynomial $3t+1$. The locus $\Gamma_A$ is irreducible because it is dominated by $\vert\ell\vert\times\vert\xi\vert$.
\end{example}

\begin{example}
\label{eB}
Arguing as in the previous remark we deduce that each pair of elements in $\vert \xi\vert$ and $\vert \ell-e\vert$ intersects along a curve $B$ which is isomorphic to $\p1$: notice that $Bh=2$, hence the Hilbert polynomial is $\chi(\cO_{B}(t))=2t+1$. The class of $B$ in $A^2(F)$ is  $\ell\xi-e\xi$
and $\cI_{B\vert F}$ fits into the Koszul--type resolution
$$
0\longrightarrow\cO_{F}(-\ell-\xi+e)\longrightarrow\cO_{F}(-\ell+e)\oplus\cO_{F}(-\xi)\longrightarrow\cI_{B\vert F}\longrightarrow0,
$$
hence $\cN_{B\vert F}\cong\cO_{\p1}^{\oplus2}$. We have $B\not\subseteq E$, because the cohomology of the above exact sequence tensored by $\cO_F(e)$ and equalities \eqref{Kuenneth} return $h^0\big(F,\cI_{B\vert F}(e)\big)=0$.

We denote by $\Gamma_{B}$ the locus of curves as above inside the Hilbert scheme of curves on $F$ with Hilbert polynomial $2t+1$. The locus $\Gamma_B$ is irreducible because it is dominated by $\vert\ell-e\vert\times\vert\xi\vert$.
\end{example}

\begin{example}
\label{eC}
Each curve $C$ whose class in $A^2(F)$ is $\ell^2$ is isomorphic to $\p1$, because it is a fibre of the projection $p$. We have that $Ch=2$, $\chi(\cO_C(t))=2t+1$ and the ideal sheaf $\cI_{C\vert F}$ fits into the Koszul--type resolution
$$
0\longrightarrow\cO_{F}(-2\ell)\longrightarrow\cO_{F}(-\ell)^{\oplus2}\longrightarrow\cI_{C\vert F}\longrightarrow0,
$$
hence $\cN_{C\vert F}\cong\cO_{\p1}^{\oplus2}$. Notice that $C$ is the inverse image of a point on ${\bF_1}$: hence $C\cap E=\emptyset$ for each general $C$, hence $C\not\subseteq E$.

Let $\Gamma_C$ be the locus of curves as above inside the Hilbert scheme of curves on $F$ with Hilbert polynomial $2t+1$. The locus $\Gamma_C$ is irreducible because it is dominated by $\vert\ell\vert^{\times2}$.
\end{example}

\section{Monads for instanton bundles on $F$}
\label{sMonad}
In this section, for each instanton bundle $\cE$ on $F$, we construct a monad whose cohomology is exactly $\cE$.

We refer the reader to the paper \cite{G--K} and the references therein for the results that we list below in a very concise way. Let $D^b (X)$ be the bounded derived category of coherent sheaves on a smooth variety $X$. For each sheaf $\cA$ and integer $k$ we denote by $\cA[k]$ the trivial complex defined as
$$
\cA[k]^i=\left\lbrace\begin{array}{ll} 
0\quad&\text{if $i\ne k$,}\\
\cA\quad&\text{if $i=k$.}
\end{array}\right.
$$
with the trivial differentials: we will often omit $[0]$ in the notation. A coherent sheaf $\cA$ is {\sl exceptional} as trivial complex in $D^b(X)$ if
$$
\Ext_X^k\big(\cA, \cA\big) = \left\lbrace\begin{array}{ll} 
0\quad&\text{if $k\ne 0$,}\\
\bC\quad&\text{if $k=0$.}
\end{array}\right.
$$
 An ordered set of exceptional objects $(\cA_0,\dots, \cA_s)$ is an {\sl exceptional collection} if $\Ext_{X}^k\big(\cA_i, \cA_j\big) = 0$ when $i > j$ and $k\in\bZ$. An exceptional collection is {\sl full} if it generates $D^b(X)$. An exceptional collection is {\sl strong} if $
\Ext_{X}^k\big(\cA_i, \cA_j\big) = 0$ when $i < j$ and $k\in\bZ\setminus\{\ 0\ \}$.

If $(\cA_0[k_0],\dots,\cA_s[k_s])$ is full and exceptional, then there exists a unique collection $(\cB_0,\dots,\cB_s )$ satisfying 
$$
\Ext_{X}^{k-k_i}\big(\cA_i,\cB_j\big)=\left\lbrace\begin{array}{ll} 
\bC\quad&\text{if $i+j= s$ and $i=k$,}\\
0\quad&\text{otherwise,}
\end{array}\right.
$$
which is called the {\sl right dual collection} of $(\cA_0[k_0],\dots,\cA_s[k_s])$ (see \cite[Proposition 2.6.1]{G--K}). 
\begin{lemma}
The collection
\begin{align*}
(\cF_0[v_0]&,\dots,\cF_7[v_7]):=(\cO_F[v_0],\cO_F(\ell-e)[v_1],\cO_F(e)[v_2],\\
&\cO_F(\ell)[v_3],\cO_F(\xi)[v_4],\cO_F(\ell-e+\xi)[v_5],\cO_F(e+\xi)[v_6],\cO_F(\ell+\xi)[v_7])
\end{align*}
where $v_0=v_1=0$, $v_2=v_3=1$, $v_4=v_5=3$, $v_6=v_7=4$, is exceptional and full. Its right dual collection is 
\begin{align*}
(\cG_0&,\dots,\cG_7):=(\cO_F(-2\ell+e-\xi),\cO_F(-\ell-\xi),\\
&\cO_F(-\ell+e-\xi),\cO_F(-\xi),\cO_F(-2\ell+e),\cO_F(-\ell),\cO_F(-\ell+e),\cO_F)
\end{align*}
and it is strong.
\end{lemma}
\begin{proof}
The projection $p\colon F\to {\bF_1}$ induces an isomorphism $F\cong\bP(\cO_{\bF_1}^{\oplus2})\to {\bF_1}$ identifying $\cO_{\bP(\cO_{\bF_1}^{\oplus2})}(1)\cong\cO_F(\xi)$. Moreover, ${\bF_1}\cong\bP(\cO_{\p1}\oplus\cO_{\p1}(-1))$. 

Thus the assertion on $(\cF_0[v_0],\dots,\cF_7[v_7])$ follows by applying \cite[Corollary 2.7]{Orl} twice, first to $\pi_0$ and then to $p$. We have $\Ext_{X}^{k}\big(\cF_i,\cG_j\big)\cong H^k\big(F,\cF_i^\vee\otimes\cG_j\big)$, because  $\cG_i$ and $\cF_j$ are locally free sheaves. Thus a direct computation of the latter space via equalities \eqref{Kuenneth} yields all the assertions on $(\cG_0,\dots,\cG_7)$.
\end{proof}

We are ready to prove the main result of this section.

\begin{lemma}
\label{lAO}
Let $\cE$ be an instanton bundle on $F$. Then $\cE$  is the cohomology in degree $0$ of a complex ${\cC}^\bullet$ with $i^{th}$--module
$$
{\cC}^i:=\bigoplus_{p+q=i}\Ext_F^{q-v_p}\big(\cF_{-p},\cE\big)\otimes \cG_{7+p}
$$
\end{lemma}
\begin{proof}
The statement is equivalent to the assertion that the complex $\cE[0]$ is quasi--isomorphic to the complex ${\cC}^\bullet$. But such an assertion is \cite[Section 2.7.3]{G--K}.
\end{proof}

Thus, in order to prove Theorem \ref{tMonad}, it suffices to compute the dimensions $e^{p,q}:=h^{q-v_p}\big(F,\cE\otimes\cF_{-p}^\vee\big)$ of the vector spaces $E^{p,q}_1=\Ext_F^{q-v_p}\big(\cF_{-p},\cE\big)$, where $-7\le p\le0$ and $0\le q\le 7$.

To this purpose we will repeatedly use Lemma \ref{lHoppe} in the following form: the rank $2$ bundle $\cE$ on $F$ is $\mu$--semistable if and only if $h^0\big(F,\cE(-a\ell+be-c\xi)\big)=0$ whenever
\begin{equation}
\label{Hoppe}
3a+2c+6\ge b.
\end{equation}

\begin{proposition}
\label{pTable}
Let $\cE$ be an instanton bundle on $F$.

If $c_2(\cE)=\alpha\ell\xi-\beta e\xi +\gamma\ell^2$ and 
$$
\delta:= h^1\big(F,\cE(-e)\big),\qquad \epsilon:=h^1\big(F,\cE(-e-\xi)\big),
$$
then $e^{p,q}$ is the number in position $(p,q)$ in the following table.
\begin{table}[H]
\begin{tiny}
\centering
\bgroup
\def\arraystretch{1.5}
\begin{tabular}{ccccccccc}
\cline{1-8}
\multicolumn{1}{|c|}{0} & \multicolumn{1}{c|}{0} & \multicolumn{1}{c|}{0} & \multicolumn{1}{c|}{0} & \multicolumn{1}{c|}{0} & \multicolumn{1}{c|}{0} & \multicolumn{1}{c|}{0} & \multicolumn{1}{c|}{0} & $q=7$ \\ 
\cline{1-8}
\multicolumn{1}{|c|}{$\alpha+\gamma-6$} & \multicolumn{1}{c|}{$\beta+\gamma+\epsilon-2$} & \multicolumn{1}{c|}{0} & \multicolumn{1}{c|}{0} & \multicolumn{1}{c|}{0} & \multicolumn{1}{c|}{0} & \multicolumn{1}{c|}{0} & \multicolumn{1}{c|}{0} & $q=6$ \\ 
\cline{1-8}
\multicolumn{1}{|c|}{0} & \multicolumn{1}{c|}{$\epsilon$} & \multicolumn{1}{c|}{$\alpha-\beta+\gamma-4$} & \multicolumn{1}{c|}{$\gamma-2$} & \multicolumn{1}{c|}{0} & \multicolumn{1}{c|}{0} & \multicolumn{1}{c|}{0} & \multicolumn{1}{c|}{0} & $q=5$ \\ 
\cline{1-8}
\multicolumn{1}{|c|}{0} & \multicolumn{1}{c|}{0} &\multicolumn{1}{c|}{0} & \multicolumn{1}{c|}{0} & \multicolumn{1}{c|}{0} & \multicolumn{1}{c|}{0} & \multicolumn{1}{c|}{0} & \multicolumn{1}{c|}{0} & $q=4$ \\ 
\cline{1-8}
\multicolumn{1}{|c|}{0} & \multicolumn{1}{c|}{0} &\multicolumn{1}{c|}{0} & \multicolumn{1}{c|}{0} & \multicolumn{1}{c|}{$\alpha-3$} & \multicolumn{1}{c|}{$\beta+\delta-1$} & \multicolumn{1}{c|}{0} & \multicolumn{1}{c|}{0} & $q=3$ \\ 
\cline{1-8}
\multicolumn{1}{|c|}{0} & \multicolumn{1}{c|}{0} &\multicolumn{1}{c|}{0} & \multicolumn{1}{c|}{0} & \multicolumn{1}{c|}{0} & \multicolumn{1}{c|}{$\delta$} & \multicolumn{1}{c|}{$\alpha-\beta-2$} & \multicolumn{1}{c|}{0} & $q=2$ \\ 
\cline{1-8}
\multicolumn{1}{|c|}{0} & \multicolumn{1}{c|}{0} & \multicolumn{1}{c|}{0} & \multicolumn{1}{c|}{0} & \multicolumn{1}{c|}{0} & \multicolumn{1}{c|}{0} & \multicolumn{1}{c|}{0} & \multicolumn{1}{c|}{0}& $q=1$ \\ 
\cline{1-8}
\multicolumn{1}{|c|}{0} & \multicolumn{1}{c|}{0} & \multicolumn{1}{c|}{0} &\multicolumn{1}{|c|}{0} & \multicolumn{1}{c|}{0} & \multicolumn{1}{c|}{0} & \multicolumn{1}{c|}{0} & \multicolumn{1}{c|}{0} & $q=0$ \\ 
\cline{1-8}
$p=-7$ & $p=-6$ & $p=-5$ & $p=-4$ & $p=-3$ & $p=-2$ & $p=-1$ & $p=0$
\end{tabular}
\egroup
\caption{The values of $e^{p,q}$}
\end{tiny}
\end{table}
\end{proposition}
\begin{proof}
It is clear that $e^{p,q}=0$ if  $p=0,-1$ and $q\ge 4$, or $p=-2,-3$ and $q=0,5,6,7$, or $p=-4,-5$ and $q=0,1,2,7$, or $p=-6,-7$ and $q\le3$. 
Lemma \ref{lNatural} yields $e^{0,q}=0$ also for $q\le3$. Thus $e^{p,q}=0$ if  $p=-1$ and $q=0$, or $p=-2,-3$ and $q=1$, or $p=-4,-5$ and $q=3$, or $p=-6,-7$ and $q=4$, thanks to Lemma \ref{lHoppe} and inequality \eqref{Hoppe}. 

Equality \eqref{Serre1}, Lemma \ref{lHoppe} and inequality \eqref{Hoppe} imply that $e^{p,q}=0$ if  $p=-1$ and $q=3$, or $p=-2,-3$ and $q=4$, or $p=-4,-5$ and $q=6$, or $p=-6,-7$ and $q=7$. 

We have $(\ell-e)^2=\ell e=0$ on ${\bF_1}$. Thus, the Koszul complexes associated to general pairs  in $H^0\big({\bF_1},\cO_{\bF_1}(\ell-e)\big)^{\oplus2}$ and $H^0\big({\bF_1},\cO_{\bF_1}(\ell)\big)\oplus H^0\big({\bF_1},\cO_{\bF_1}(e)\big)$ induce the exact sequences
\begin{gather*}
0\longrightarrow \cE(e-\ell)\longrightarrow \cE^{\oplus2}\longrightarrow \cE(\ell-e)\longrightarrow 0,\\
0\longrightarrow \cE(-\ell)\longrightarrow \cE\oplus\cE(e-\ell)\longrightarrow \cE(e)\longrightarrow 0.
\end{gather*}
Equality \eqref{Serre1} implies $h^0\big(F,\cE(\ell-e)\big)=e^{-1,3}$. As pointed out above the number on the right is zero, hence the cohomology of the first sequence above yields $e^{-1,1}=0$, also thanks to Lemma \ref{lNatural}. Thus the cohomology of the second sequence above also implies $e^{-3,2}=0$, thanks to $e^{-2,4}=0$ and equality \eqref{Serre1}.

Equality \eqref{Serre1} implies $h^0\big(F,\cE(\xi)\big)=e^{-4,6}$, which we know to be zero. Thus the cohomology of sequence 
\begin{equation}
\label{seqEuler}
0\longrightarrow\cO_{F}(-\xi)\longrightarrow\cO_{F}^{\oplus2}\longrightarrow\cO_{F}(\xi)\longrightarrow0
\end{equation}
 tensored by $\cE$, Lemma \ref{lNatural} and equality \eqref{Serre1} return $e^{-4,4}=h^1\big(F,\cE(-\xi)\big)=0$.

We have $h^0\big(F,\cE(-\ell+e+\xi)\big)=h^0\big(F,\cE(-\ell+\xi)\big)=0$ again by Lemma \ref{lHoppe} and inequality \eqref{Hoppe}: the cohomology of sequence \eqref{seqEuler} tensored by $\cE(e-\ell)$ and $\cE(-\ell)$ then yields respectively $e^{-5,4}=h^0\big(F,\cE(-\ell+e+\xi)\big)=0$ and $e^{-7,5}=h^0\big(F,\cE(-\ell+\xi)\big)=0$. 

The remaining values of $e^{p,q}$ can be computed using equality \eqref{RRgeneral} and the above vanishings.
\end{proof}

\begin{remark}
\label{rInequalities}
As an immediate consequence of the non--negativity of the $e^{p,q}$'s in Table 1 we deduce inequalities \eqref{BoundInstanton1} and \eqref{BoundInstanton2}. 
\end{remark}

Moreover, we also have the following result.

\begin{corollary}
\label{cBound}
Let $\cE$ be an instanton on $F$ with $c_2(\cE)=\alpha\ell\xi-\beta e\xi+\gamma\ell^2$, then $2\delta\ge\epsilon$.
\end{corollary}
\begin{proof}
It suffices to compute the cohomology of sequence \eqref{seqEuler} tensored by $\cE(-e)$, taking into account that $h^0\big(F,\cE(e-\xi)\big)=0$ by Lemma \ref{lHoppe} and inequality \eqref{Hoppe}. 
\end{proof}

We are ready to prove Theorem \ref{tMonad} stated in the introduction. Recall that
\begin{gather*}
\cC^{-1}:=\cO_{F}(-2\ell+e-\xi)^{\oplus\alpha+\gamma-6}\oplus\cO_{F}(-\ell-\xi)^{\oplus\epsilon},\\
\begin{align*}
\cC^0:=\cO_{F}(-\ell-\xi)^{\oplus\beta+\gamma+\epsilon-2}&\oplus\cO_{F}(-\ell+e-\xi)^{\oplus\alpha-\beta+\gamma-4}\oplus\\
&\oplus\cO_F(-2\ell+e)^{\oplus\alpha-3}\oplus\cO_F(-\ell)^{\oplus\delta},
\end{align*}\\
\cC^1:=\cO_{F}(-\xi)^{\oplus\gamma-2}\oplus\cO_{F}(-\ell)^{\oplus\beta+\delta-1}\oplus\cO_{F}(-\ell+e)^{\oplus\alpha-\beta-2}.
\end{gather*}

\medbreak
\noindent{\it Proof of Theorem \ref{tMonad}.}
The first part of the statement can be obtained by combining Lemma \ref{lAO} and Proposition \ref{pTable}.

Conversely, assume that the cohomology $\cE$ of monad \eqref{Monad}  is a $\mu$--semistable vector bundle of rank $2$, so   $h^0\big(F,\cE\big)=0$. It is easy to check that
\begin{gather*}
c_1(\cE)=-3\ell+e-2\xi,\qquad
c_2(\cE)=\alpha\ell\xi-\beta e\xi +\gamma \ell^2.
\end{gather*}
We have to check the other bounds on the cohomology of $\cE$. To this purpose consider the two short exact sequences
\begin{equation}
\label{Display}
\begin{gathered}
0\longrightarrow \cK\longrightarrow \cC^0\longrightarrow\cC^{1}\longrightarrow0,\\
0\longrightarrow \cC^{-1}\longrightarrow \cK\longrightarrow\cE\longrightarrow0.
\end{gathered}
\end{equation}
Equalities \eqref{Kuenneth} and the cohomology of  sequences \eqref{Display} tensored by $\cO_F(-e)$ imply
$$
h^1\big({F},\cE(-e)\big)\le (\alpha-3)h^1\big({\bF_1},\cO_{\bF_1}(-2\ell)\big)+\delta h^1\big({\bF_1},\cO_{\bF_1}(-\ell-e)\big).
$$
Trivially $h^0\big({\bF_1},\cO_{\bF_1}(-2\ell)\big)=h^0\big({\bF_1},\cO_{\bF_1}(-\ell-e)\big)=0$. Moreover
\begin{gather*}
h^2\big({\bF_1},\cO_{\bF_1}(-2\ell)\big)=h^0\big({\bF_1},\cO_{\bF_1}(-\ell+e)\big)=0,\\
h^2\big({\bF_1},\cO_{\bF_1}(-\ell-e)\big)=h^0\big({\bF_1},\cO_{\bF_1}(-2\ell+2e)\big)=0.
\end{gather*}
Thus equality \eqref{RRsurface} yields
\begin{gather*}
-h^1\big({\bF_1},\cO_{\bF_1}(-2\ell)\big)=\chi(\cO_{\bF_1}(-2\ell))=0, \\
 -h^1\big({\bF_1},\cO_{\bF_1}(-\ell-e)\big)=\chi(\cO_{\bF_1}(-\ell-e))=-1,
\end{gather*}
whence we deduce $h^1\big({F},\cE(-e)\big)\le \delta$. Essentially the same argument also yields $h^1\big({F},\cE(-\ell-e)\big)\le \epsilon$.

Now let $D\in\vert a\ell-be+c\xi\vert$ be effective. Thus we obtain $a=D\ell\xi\ge0$, $c=D\ell^2\ge0$ and $a-b=D(\ell-e)\xi\ge0$, as an immediate consequence of Remark \ref{rNEF}. Assume $b<0$. Then $b=De\xi$, hence all the lines with class $e\xi$ are contained in $D$. Since such lines cover the whole exceptional divisor $E$, it follows that $E\subseteq D$. Thus, if $D$ is integral, either $D=E$ or $D\ne E$ and $b\ge0$.

If $a,b,c\ge0$, then again the cohomology of sequences \eqref{Display} implies
\begin{align*}
h^1\big(F,\cE(-D)\big)&\le (\alpha+\gamma-6) h^2\big(F,\cO_F(-(a+2)\ell+(b+1)e-(c+1)\xi)\big)+\\
&+\epsilon h^2\big(F,\cO_F(-(a+1)\ell+be-(c+1)\xi)\big).
\end{align*}
If $c=0$, then equalities \eqref{Kuenneth} automatically imply $h^1\big(F,\cE(-D)\big)=0$. In order to prove that $h^1\big(F,\cE(-D)\big)=0$ when $c\ge1$, it suffices to check that
\begin{equation}
\label{h^1vanishing}
h^1\big({\bF_1},\cO_{\bF_1}(-(a+2)\ell+(b+1)e)\big)=h^1\big({\bF_1},\cO_{\bF_1}(-(a+1)\ell+be)\big)=0
\end{equation}
again thanks to equalities \eqref{Kuenneth}. Notice that
$$
h^0\big({\bF_1},\cO_{\bF_1}(-(a+2)\ell+(b+1)e)\big)=h^0\big({\bF_1},\cO_{\bF_1}(-(a+1)\ell+be)\big)=0,
$$
because $a\ge0$. Equalities \eqref{Serre}, \eqref{dimh^0} and \eqref{RRsurface} yield 
\begin{gather*}
\begin{align*}
h^2\big({\bF_1},\cO_{\bF_1}(-(a+2)\ell+(b+1)e)\big)&=h^0\big({\bF_1},\cO_{\bF_1}((a-1)\ell-be)\big)=\\
&=\chi(\cO_{\bF_1}(-(a+2)\ell+(b+1)e)),
\end{align*}\\
\begin{align*}
h^2\big({\bF_1},\cO_{\bF_1}(-(a+1)\ell+be)\big)&=h^0\big({\bF_1},\cO_{\bF_1}((a-2)\ell-(b-1)e)\big)=\\
&=\chi(\cO_{\bF_1}(-(a+1)\ell+be))
\end{align*}
\end{gather*}
hence equalities \eqref{h^1vanishing} hold true. In particular $h^1\big(F,\cE\big)=0$.
\qed
\medbreak

\begin{remark}
\label{rConfront}
Each monad on $F$ can be also viewed as a family of monads on $\bF_1$ parameterized by $\p1$.  

Let $\cE$ be an instanton bundle $\cE$ with $c_2(\cE)=\alpha\ell\xi-\beta e\xi+\gamma\ell^2$. If $S\cong\bF_1$ is any fibre of $\pi$, then $\cE_S:=\cE\otimes\cO_S$ is a vector bundle satisfying $c_1(\cE_S)=-3\ell+e$ and $c_2(\cE_S)=\gamma$. Since $\cE_S$ is a vector bundle, then the restrictions of the sequences \eqref{Display} to $S$ yields the monad
\begin{align*}
0&\longrightarrow\cO_S(-2\ell+e)^{\oplus\alpha+\gamma-6}\oplus\cO_S(-\ell)^{\oplus\epsilon}\longrightarrow\\
&\longrightarrow\cO_S(-\ell)^{\oplus\beta+\gamma+\delta+\epsilon-2}\oplus\cO_S(-\ell+e)^{\oplus\alpha-\beta+\gamma-4}\oplus\cO_S(-2\ell+e)^{\oplus\alpha-3}\longrightarrow\\
&\longrightarrow\cO_S^{\oplus\gamma-2}\oplus\cO_S(-\ell)^{\oplus\beta+\delta-1}\oplus\cO_S(-\ell+e)^{\oplus\alpha-\beta-2}\longrightarrow0
\end{align*}
whose cohomology is  $\cE_S$. 

Monads on $\bF_1$ have been widely studied: e.g. see \cite{King, Buch1, Buch2}. In particular, if $\sigma_*\cE_S$ is semistable, then \cite[Proposition 1.10]{Buch2} implies that $\cE_S$ is the cohomology of a monad of the form
\begin{equation*}
0\longrightarrow\cK\longrightarrow\cO_S(-\ell)^{\oplus k}\longrightarrow\mathcal L\longrightarrow0
\end{equation*}
where $\cK$ and $\mathcal L$ are direct sums of suitable line bundles on $S\cong\bF_1$ and $k$ depends only on $\gamma$.

The two monads above are trivially different. Moreover, the cohomologies of the restrictions to $S$ of the sequences \eqref{Display} yield $h^0\big(S,\cE_S(\ell)\big)\ge3-\gamma$. The Leray spectral sequence $E_2^{p,q}:=H^p\big(\p2,R^q\sigma_*\cE_S\otimes\cO_{\p2}(1)\big)$ abuts to $H^{p+q}\big(S,\cE_S(\ell)\big)$, hence
$$
h^0\big(\p2,\sigma_*\cE_S\otimes\cO_{\p2}(1)\big)\ge3-\gamma.
$$
Lemma \ref{lHoppe} then implies that $\sigma_*\cE_S$ is not $\mu$--semistable if $\gamma=2$. It follows that $\sigma_*\cE_S$ is not even semistable in this case.
\end{remark}

Let us conclude this section with a remark on earnest instanton bundles. Let us recall this notion (see \cite{C--C--G--M} for more details): an instanton bundle $\cE$ on a Fano threefold is called {\sl earnest} if $h^1\big(X,\cE(-q_Xh-D)\big)=0$ whenever $\vert D\vert\ne\emptyset$ contains smooth integral elements. The following corollary is an immediate consequence of Theorem \ref{tMonad} and Corollary \ref{cBound}.

\begin{corollary}
\label{cExotic}
Let $\cE$ be an instanton on $F$. Then $\cE$ is earnest if and only if $h^1\big({F},\cE(-e)\big)=0$.
\end{corollary}

\section{Existence of instanton bundles}
\label{sInstanton}

In this section we show that  instanton bundles on $F$ exist for each admissible value of the second Chern class. As pointed in the introduction it is easy to give examples of decomposable instanton bundles on $F$.

\begin{example}
\label{eDecomposable}
Let $\cA:=\cO_F(-2\ell)$ and $\cB:=\cO_F(-\ell+e-2\xi)$. If $\cE:=\cA\oplus\cB$, then $c_1(\cE)=-h$ and  equalities \eqref{Kuenneth} yield $h^i\big(F,\cE\big)=0$ for $i=0,1$. Moreover, $\mu(\cA)=\mu(\cB)=-24$, hence $\cE$ is also $\mu$--semistable. Thus $\cE$ is a decomposable instanton bundle.
\end{example}

In order to construct indecomposable instanton bundles we need two different constructions, one for the earnest bundles and the other for non--earnest ones.

\begin{construction}
\label{conInstanton}
Let $\alpha,\beta,\gamma\in\bZ$ satisfy inequalities \eqref{BoundInstanton1}. If $\beta\ge1$, thanks to Examples \ref{eA}, \ref{eB} and \ref{eC} we can take pairwise disjoint curves
$$
A_1,\dots, A_{\alpha-2-\beta}\in \Gamma_A,\qquad B_1,\dots,B_{\beta-1}\in\Gamma_{B},\qquad C_1,\dots,C_{\gamma-2}\in\Gamma_C.
$$
If $\beta\le0$, thanks to Examples  \ref{eA}, \ref{eLine} and \ref{eC} we can take pairwise disjoint curves
$$
A_1,\dots, A_{\alpha-3}\in \Gamma_A,\qquad B_1,\dots,B_{1-\beta}\in\Lambda_{E},\qquad C_1,\dots,C_{\gamma-2}\in\Gamma_C.
$$
We use Theorem \ref{tSerre} for constructing an instanton bundle from the  union $Z$ of such pairwise disjoint curves. The curve $Z$ is  smooth and 
$$
\det(\cN_{Z\vert F})\cong\left\lbrace\begin{array}{ll} 
\cO_{F}(\ell- e)\otimes\cO_Z\quad&\text{if $\beta\ge1$,}\\
\cO_{F}(\ell+e)\otimes\cO_Z\quad&\text{if $\beta\le0$,}
\end{array}\right.
$$
because all the connected components of $Z$ are isomorphic to $\p1$ and
\begin{equation*}
(\ell\pm e)\ell\xi=1,\qquad (\ell- e)(\ell\xi-e\xi)=(\ell\pm e)\ell^2=0,\qquad (\ell+ e)e\xi=-1.
\end{equation*}
Thanks to equalities \eqref{Kuenneth} and \eqref{Serre} on ${\bF_1}$ we have
\begin{gather*}
h^2\big({F},\cO_{F}(-\ell+e)\big)=h^2\big({{\bF_1}},\cO_{{\bF_1}}(-\ell+e)\big)=h^0\big({{\bF_1}},\cO_{{\bF_1}}(-2\ell)\big)=0\\
h^2\big({F},\cO_{F}(-\ell-e)\big)=h^2\big({{\bF_1}},\cO_{{\bF_1}}(-\ell-e)\big)=h^0\big({{\bF_1}},\cO_{{\bF_1}}(2e-2\ell)\big)=0.
\end{gather*}
If $\beta\ge1$, then there is a vector bundle $\cF$ on $F$ such that $c_1(\cF)=\ell-e$, $c_2(\cF)=Z$ and fitting into sequence \eqref{seqSerre}. Thus $\cE:=\cF(-2\ell+e-\xi)$ fits into 
\begin{equation}
\label{seqEarnest}
0\longrightarrow\cO_{F}(-2\ell+e-\xi)\longrightarrow\cE\longrightarrow\cI_{Z\vert F}(-\ell-\xi)\longrightarrow0.
\end{equation}
If $\beta\le0$, then there is a vector bundle $\cF$ on $F$ such that $c_1(\cF)=\ell+e$, $c_2(\cF)=Z$ and fitting into sequence \eqref{seqSerre}. Thus $\cE:=\cF(-2\ell-\xi)$ fits into
\begin{equation}
\label{seqNonEarnest}
0\longrightarrow\cO_{F}(-2\ell-\xi)\longrightarrow\cE\longrightarrow\cI_{Z\vert F}(-\ell+e-\xi)\longrightarrow0.
\end{equation}
\end{construction}

It is easy to check that $c_1(\cE)=-h$ and $c_2(\cE)=\alpha\ell\xi-\beta e\xi+\gamma\ell^2$. 
\begin{remark}
\label{rBound}
Thanks to  inequalities \eqref{BoundInstanton1} we deduce $c_2(\cE)h\ge14$ if $\beta\ge1$ and $c_2(\cE)h\ge15$ if $\beta\le0$.
\end{remark}

We can now prove Theorem \ref{tConstruction} stated in the introduction.

\medbreak
\noindent{\it Proof of Theorem \ref{tConstruction}.}
First we consider the case $\beta\ge1$. 

We start the proof by dealing with the $\mu$--stability of $\cE$ with the help of Lemma \ref{lHoppe} and inequality \eqref{Hoppe}. The cohomology of sequence \eqref{seqEarnest} tensored by $\cO_{F}(-a\ell+be-c\xi)$ with $3a+2c+6\ge b$ yields the inequalities
\begin{align*}
h^0\big(F,\cE(-a\ell+be-c\xi)\big)\le  u+w\le v+w
\end{align*}
where
\begin{gather*}
u:=h^0\big(F,\cI_{Z\vert F}(-(a+1)\ell+be-(c+1)\xi)\big),\\
v:=h^0\big(F,\cO_{F}(-(a+1)\ell+be-(c+1)\xi)\big),\\
w:=h^0\big(F,\cO_{F}(-(a+2)\ell+be-(c+1)\xi)\big).
\end{gather*}
Thanks to equalities \eqref{Kuenneth}
\begin{gather*}
v=h^0\big({\bF_1},\cO_{{\bF_1}}(-(a+1)\ell+be)\big)h^0\big(\p1,\cO_{\p1}(-c-1)\big),\\
w=h^0\big({\bF_1},\cO_{{\bF_1}}(-(a+2)\ell+be)\big)h^0\big(\p1,\cO_{\p1}(-c-1)\big).
\end{gather*}

We claim that $w=0$. This is true if $h^0\big({\bF_1},\cO_{{\bF_1}}(-(a+2)\ell+be)\big)=0$. If $h^0\big({\bF_1},\cO_{{\bF_1}}(-(a+2)\ell+be)\big)\ge1$, then $b\ge 3a+6$ because $\cO_{\bF_1}(3\ell-e)$ is ample. Thus the inequality $3a+2c+6\ge b$ implies $c\ge0$, hence $h^0\big(\p1,\cO_{\p1}(-c-1)\big)=0$ and $w=0$

We claim that $u=0$. This is  true if $h^0\big({\bF_1},\cO_{{\bF_1}}(-(a+1)\ell+be)\big)=0$. If $h^0\big({\bF_1},\cO_{{\bF_1}}(-(a+1)\ell+be)\big)\ge1$, then $a\le-1$ because $\cO_{\bF_1}(\ell)$ is globally generated. The same argument used to prove that $w=0$ implies $c=-1$. Thus the inequality $3a+2c+6\ge b$ forces $a=-1$, $b\ge0$ and, consequently, $u=h^0\big(F,\cI_{Z\vert F}(be)\big)$. The linear system $\vert be\vert$ on $F$ contains a single divisor with support $E$, because the same is true for divisors in the linear system $\vert be\vert$ on ${\bF_1}$. As pointed out in Examples \ref{eA}, \ref{eB}, \ref{eC} no component of $Z$ is contained in $E$, it follows that $u=0$. Thus the proof of the $\mu$--stability of $\cE$ is complete.

In particular, the above discussion also implies $h^0\big(F,\cE\big)=0$. We will show below that $h^1\big(F,\cE(-D)\big)=0$ for each divisor $D\subseteq F$ with $h^0\big(F,\cO_F(D)\big)\ne0$. Thus we also deduce $h^1\big(F,\cE\big)=0$. Since $c_1(\cE)=-h$ by construction, it follows that $\cE$ is earnest instanton bundle. 

In order to check the aforementioned vanishing $h^1\big(F,\cE(-D)\big)=0$, assume that $\cO_F(D)\cong\cO_F(a\ell-be+c\xi)$. Since $\cO_F(\ell)$, $\cO_F(\ell-e_k)$ and $\cO_F(\xi)$ are globally generated, it follows that $a=D\ell\xi\ge0$, $a-b=D(\ell\xi-e\xi)\ge0$ and $c=D\ell^2\ge0$. 

The cohomology of sequence \eqref{seqEarnest} tensored by $\cO_F(-D)$ returns
\begin{align*}
h^1\big(F,\cE(-D)\big)&\le h^1\big(F,\cO_F(-(a+2)\ell+be-(c+1)\xi)\big)+\\
&+h^1\big(F,\cI_{Z\vert F}(-(a+1)\ell+be-(c+1)\xi)\big).
\end{align*}
Thus it suffices to check the vanishing of the summands on the right for each choice of integers $a,c\ge0$ and $b\le a$.

Equalities \eqref{Kuenneth} and the aforementioned restrictions imply that the first righthand summand vanishes. Moreover, a similar argument applied to the cohomology of sequence
\begin{equation}
\label{seqIdeal}
0\longrightarrow \cI_{Z\vert F}\longrightarrow \cO_{F}\longrightarrow \cO_{Z}\longrightarrow0
\end{equation}
tensored by $\cO_F(-\ell-\xi-D)$ yields
\begin{align*}
h^1\big(F,\cI_{Z\vert F}(-(a+1)\ell&+be-(c+1)\xi)\big)=\\
&=h^0\big(Z,\cO_F(-(a+1)\ell+be-(c+1)\xi)\otimes\cO_Z\big).
\end{align*}
The definition of $Z$ combined with the restrictions above on $a$, $b$ and $c$ returns also $h^1\big(F,\cI_{Z\vert F}(-(a+1)\ell+be-(c+1)\xi)\big)=0$.

The restriction of sequence \eqref{seqEarnest} to each line $L\subseteq F$ not intersecting $Z$ is
$$
0\longrightarrow\cO_{\p1}(-1)\longrightarrow\cE\otimes\cO_L\longrightarrow\cO_{\p1}\longrightarrow0:
$$
hence $\cE\otimes\cO_L\cong\cO_{\p1}\oplus\cO_{\p1}(-1)$ for each general line $L\subseteq F$.

Now we turn our attention to the case $\beta\le 0$. The proof in this case runs along the same lines of the one in the case $\beta\ge1$. In particular, the assertion of  $\cE\otimes\cO_L$ is immediate to prove. 

In order to prove the $\mu$--stability of $\cE$, it suffices to check that
\begin{gather*}
w:=h^0\big(F,\cO_{F}(-(a+2)\ell+be-(c+1)\xi)\big),\\
u:=h^0\big(F,\cI_{Z\vert F}(-(a+1)\ell+(b+1)e-(c+1)\xi)\big),
\end{gather*}
vanish when $3a+2c+6\ge b$. Equalities \eqref{Kuenneth} imply that $c\le -1$ if either $w$ or $u$ is non--zero as in the proof of the case $\beta\ge1$. If $w\ne0$, then we obtain $c\ge0$, a contradiction. If $u\ne0$, then $v:=h^0\big(Z,\cO_F(-(a+1)\ell+(b+1)e-(c+1)\xi)\big)\ne0$. Since $\cO_{\bF_1}(3\ell-e)$ is ample, it follows that $b\ge3a+2$ thanks to the Nakai--Moishezon criterion, hence $c\ge-2$, i.e. $c=-1,-2$. 

If $c=-2$, then $b=3a+2$. Thus $\cO_{\bF_1}(-(a+1)\ell+(b+1)e)=\cO_{\bF_1}$ thanks to the Nakai--Moishezon criterion, hence $u=h^0\big(F,\cI_{Z\vert F}(\xi)\big)$. The dimension on the right vanishes if $\gamma\ge3$, because each divisor in $\vert \xi\vert$ is transversal to the curves in $\Gamma_C$. If $\gamma=2$, then $\alpha\ge4$, hence $Z$ contains curves both in $\Gamma_A$ and in $\Lambda$: for a general choice of $Z$ such curves are not contained in the same divisor  in $\vert \xi\vert$, thus again $u=0$.

If $c=-1$, then $3a+4\ge b\ge 3a+2$. If $b=3a+2$, as in the case $c=-2$ one deduces $u=h^0\big(F,\cI_{Z\vert F}\big)=0$, because $\deg(Z)\ge3$. If $b=3a+3$, then $-(a+1)\ell+(b+1)e$ is the class of a line in ${\bF_1}$, which must be $E$ thanks to Example \ref{eLine}, hence $u=h^0\big(F,\cI_{Z\vert F}(e)\big)$. Since $Z$ contains at least a curve in $\Gamma_A\cup\Gamma_C$, it follows that $u=0$ thanks to Examples \ref{eA} and \ref{eC}. If $b=3a+4$, then the linear system $\vert -(a+1)\ell+(b+1)e\vert$ on ${\bF_1}$ contains a curve of degree $2$ which can be either $2E$ or an irreducible conic. In the first case $Z\subseteq E$ which is not possible as we just proved above, hence we have only to handle the second case. The genus formula for the arithmetic genus implies that each integral conic on ${\bF_1}$ is in the linear system $\vert\ell-e\vert$, hence $u=h^0\big(F,\cI_{Z\vert F}(\ell-e)\big)$. No curves in $\Lambda$ are contained in a divisor in the linear system $\vert \ell-e\vert$ on $F$, hence $u=0$ again. Thus the proof of the $\mu$--stability is complete. 

Finally we compute $\delta:=h^1\big(F,\cE(-e)\big)$ and $\epsilon:=h^1\big(F,\cE(-e-\xi)\big)$. The cohomology of sequences \eqref{seqNonEarnest} and \eqref{seqIdeal} tensored by $\cO_F(-e-\xi)$ and $\cO_F(-\ell-2\xi)$ respectively and equalities \eqref{Kuenneth} yield
\begin{gather*}
h^1\big(F,\cI_{Z\vert F}(-\ell-2\xi)\big)-1\le \epsilon\le h^1\big(F,\cI_{Z\vert F}(-\ell-2\xi)\big),\\
h^1\big(F,\cI_{Z\vert F}(-\ell-2\xi)\big)=h^0\big(Z,\cO_F(-\ell-2\xi)\otimes\cO_Z\big).
\end{gather*}
The dimension on the right in the latter equality is $1-\beta$ because $(\ell+2\xi)\ell\xi=1$, $(\ell+2\xi)\ell^2=2$ and $(\ell+2\xi)e\xi=0$. Similarly we obtain  $\delta=1-\beta\ge1$, hence $\cE$ is not earnest by Corollary \ref{cExotic}.
\qed
\medbreak

The bundles $\cE$ obtained via Construction \ref{conInstanton} are $\mu$--stable, hence simple and indecomposable. We have $h^1\big(F,\cO_F(-\ell+e)\big)=0$ and $h^1\big(F,\cO_F(-\ell-e)\big)=1$. In the next proposition we deal with the component of the moduli space were they sit. 

\begin{proposition}
\label{pModuli}
The bundle $\cE$ obtained via Construction \ref{conInstanton} represents a smooth point in a component
$$
\cS\cI_F^0(\alpha\ell\xi-\beta e\xi+\gamma \ell^2)\subseteq\cS\cI_F(\alpha\ell\xi-\beta e\xi+\gamma \ell^2)
$$
of dimension 
$$
\dim\Ext^1_F\big(\cE,\cE\big)=6\alpha-2\beta+4\gamma-27.
$$
\end{proposition}
\begin{proof}
Thanks to Theorem \ref{tSerre}, if $\beta\le0$ the bundles are parameterized by the points of a projective bundle $\Pi$ with fibre $\p1$ over the open subset $\Sigma_-$  corresponding to smooth curves inside $\Gamma_A^{\times\alpha-3}\times\Lambda^{\times 1-\beta}\times \Gamma_C^{\times\gamma-2}$, hence there is a morphism
$$
\psi_-\colon\Pi\to \cS_F(2;0,\alpha\ell\xi-\beta e\xi+\gamma \ell^2).
$$
Thanks to Examples \ref{eLine}, \ref{eA}, \ref{eC}, we know that $\Sigma_-$ is irreducible, hence the same is true for $\Pi$. If $\beta\ge1$, Construction \ref{conInstanton} yields the existence of another morphism
$$
\psi_+\colon\Sigma_+\to \cS_F(2;0,\alpha\ell\xi-\beta e\xi+\gamma \ell^2),
$$
defined over the open irreducible subset $\Sigma_+\subseteq \Gamma_A^{\times\alpha-2-\beta}\times\Lambda^{\times \beta-1}\times \Gamma_C^{\times\gamma-2}$ corresponding to smooth curves. Since both $\Sigma_+$ and $\Pi$ are irreducible, it follows that $\cS_+:=\im(\psi_+)$ and $\cS_-:=\im(\psi_-)$ are irreducible too, hence they are contained in at least one component of $\cS\cI_F(\alpha\ell\xi-\beta e\xi+\gamma \ell^2)$. 

We show below that the points in $\cS_+$ and $\cS_-$ are smooth, hence there is actually a unique  component containing them and we denote it by $\cS\cI_F^0(\alpha\ell\xi-\beta e\xi+\gamma \ell^2)$. In order to prove the smoothness of $\cS_+$ and $\cS_-$ it suffices to check that
$$
\dim\Ext^2_F\big(\cE,\cE\big)=h^2\big(F,\cE\otimes\cE^\vee\big)=0.
$$
and, in what follows, we make the computations only for the case $\beta\le0$, because the argument in the case $\beta\ge1$ is similar.

The cohomology of sequence \eqref{seqNonEarnest} tensored by $\cE^\vee\cong\cE(h)$ and $\cO_F(\ell-e+\xi)$  yields
$$
h^2\big(F,\cE\otimes\cE^\vee\big)\le h^2\big(F,\cE\otimes\cI_{Z\vert F}(2\ell+\xi)\big)+h^2\big(F,\cE(\ell-e+\xi)\big).
$$
Equality \eqref{Serre1} implies $h^2\big(F,\cE(\ell-e+\xi)\big)=h^1\big(F,\cE(-\ell+e-\xi)\big)$ and we checked in Proposition \ref{pTable} that the latter is zero.

The cohomology of sequences \eqref{seqIdeal} and \eqref{seqNonEarnest} tensored by $\cE(2\ell+\xi)$ and $\cO_F(2\ell+\xi)$ respectively yields
$$
h^2\big(F,\cE\otimes\cI_{Z\vert F}(2\ell+\xi)\big)\le h^2\big(F,\cI_{Z\vert F}(\ell+\xi)\big)+h^1\big(Z,\cE(2\ell+\xi)\otimes\cO_Z\big).
$$
Equality \eqref{Normal} and the definition of $\cE$ imply $\cE(2\ell+\xi)\otimes\cO_Z\cong\cN_{Z\vert F}$. Since $Z$ is the non--empty disjoint union of the curves $A_j$, $B_j$, $C_j$, it is easy to check that the second summand on the right vanishes. The cohomology of sequence  \eqref{seqIdeal} tensored by $\cO_F(\ell+\xi)$ finally yields 
$$
h^2\big(F,\cI_{Z\vert F}(\ell+\xi)\big)\le h^1\big(Z,\cO_F(\ell+\xi)\otimes\cO_Z\big).
$$
Again the definition of $Z$ implies that the dimension on the right vanishes.

Finally, the value of $\dim\Ext^1_F\big(\cE,\cE\big)$ can be deduced from the above vanishing,  Lemma \ref{lExt3} and equality \eqref{Ext12}.
\end{proof}

\begin{remark}
When $\beta\le0$, the last inequality \eqref{BoundInstanton2} implies that all the points in $\cS\cI_F(\alpha\ell\xi-\beta e\xi+\gamma \ell^2)$ represent non--earnest instanton bundles. When $\beta\ge1$ the general element in $\cS\cI_F^0(\alpha\ell\xi-\beta e\xi+\gamma \ell^2)$ is earnest, by semicontinuity.
\end{remark}

\section{Minimal instanton bundles}
\label{sMinimal}
If $\cE$ is an instanton bundle on a Fano threefold $X$, then $c_2(\cE) h$ is bounded from below. Indeed
\begin{equation*}
c_2(\cE) h\ge k_X:=\left\lbrace\begin{array}{ll} 
1\quad&\text{if $i_X=4,3$,}\\
2\quad&\text{if $i_X=2$,}\\
\left\lceil\frac{\deg(X)}4\right\rceil\quad&\text{if $i_X=1$}
\end{array}\right.
\end{equation*}
(see \cite[Corollary 4.2]{C--C--G--M}). We say that $\cE$ is {\sl minimal} if $c_2(\cE) h$ is as small as possible.

On the one hand, if $X=F$, then $\deg(X)=48$. On the other hand, $c_2(\cE)h\ge14$ by Remark \ref{rBound} for each instanton bundle $\cE$ on $F$, hence the above general inequality is certainly not sharp for $F$.

Anyhow, taking either $\alpha=4$ and $\beta=\gamma=2$ or $\alpha=\gamma=3$ and $\beta=1$ in Construction \ref{conInstanton} one obtains the existence of instanton bundles for which the equality is attained: thus minimal instanton bundles satisfy $c_2(\cE)h=14$.

\medbreak
\noindent{\it Proof of Theorem \ref{tMinimal}.}
On the one hand, inequalities \eqref{BoundInstanton1} easily imply that if a minimal instanton bundle $\cE$ exists on $F$, then $c_2(\cE)$ is either $3\ell\xi- e\xi+3 \ell^2$ or $4\ell\xi-2 e\xi+2 \ell^2$. On the other hand, Theorem \ref{tConstruction} for $\beta\ge1$ implies that instanton bundles $\cE$ with charge either $4\ell\xi-2e\xi+2 \ell^2$ or $3\ell\xi-e\xi+3 \ell^2$ certainly exist.

Let $\cE$ be a minimal instanton bundle on $F$ so that $c_2(\cE)h=14$. The cohomology of 
\begin{equation*}
0\longrightarrow\cE(-\ell)\longrightarrow\cE(e-\ell)\oplus\cE\longrightarrow\cE(e)\longrightarrow0
\end{equation*}
 and equality \eqref{Serre1} yield $\delta=h^1\big(F,\cE(-e)\big)=h^2\big(F,\cE(e)\big)=0$, hence $\cE$ is earnest, thanks to Corollary \ref{cExotic}. Thus $\epsilon=h^1\big(F,\cE(-e-\xi)\big)=0$, thanks to Corollary \ref{cBound}. Moreover, Theorem \ref{tMonad} yields that $\cE$ is the kernel of a suitable map between decomposable bundles on $F$.

From now on we will assume $c_2(\cE)=4\ell\xi-2e\xi+2 \ell^2$. Theorem \ref{tMonad} implies the existence of an exact sequence of the form
\begin{equation}
\label{seqMinimalEarnest1}
0\longrightarrow \cE\longrightarrow \cO_{F}(-\ell-\xi)^{\oplus2}\oplus\cO_F(-2\ell+e)\mapright{M}\cO_{F}(-\ell)\longrightarrow 0,
\end{equation}
where $M=(m_1,m_2,m)$ where $m_1,m_2\in H^0\big(F,\cO_F(\xi)\big)$ and $m\in H^0\big(F,\cO_F(\ell-e)\big)$. Since $M$ is surjective and the zero loci of $m_i$ and $m$ intersect, it follows that $m_1,m_2$ are linearly independent. If $m=0$, then $\cE\cong \cO_F(-\ell-2\xi)\oplus\cO_F(-2\ell+e)$ which  is not $\mu$--semistable, because $\mu(\cO_F(-\ell-2\xi))=-28\ne-20=\mu(\cO_F(-2\ell-e))$. Thus $m\ne0$.

We claim that the restriction of $\cE$ to the general line $L\subseteq F$ is isomorphic to $\cO_{\p1}\oplus\cO_{\p1}(-1)$. Indeed, since $m_1,m_2$ are linearly independent, it follows that the restrictions of $m_1$ and $m_2$ cannot vanish simultaneously on $L$. Since the class of $L$ in $A^2(F)$ is $e\xi$, it follows that the restriction of sequence \eqref{seqMinimalEarnest1} implies that $\cE\otimes\cO_L\cong\cO_{\p1}\oplus\cO_{\p1}(-1)$, as claimed in the statement.

We claim that $\cE$ is aCM. To prove the claim, notice that the cohomology of the dual of sequence  \eqref{seqMinimalEarnest1} tensored by $\cO_F((t-1)h)$ yields
\begin{align*}
h^1\big(F,\cE(th)\big)&\le 2h^1\big(F,\cO_F((t-1)h+\ell+\xi)\big)+\\
&+h^1\big(F,\cO_F((t-1)h+2\ell-e)\big)+h^2\big(F,\cO_F((t-1)h+\ell)\big).
\end{align*}
Thus the assertion follows if we prove that the summands on the right vanish when $t\ge1$. Thanks to equality \eqref{Kuenneth} and the restriction $t\ge1$ such vanishings are equivalent to prove
\begin{align*}
h^1\big(\bF_1,\cO_{\bF_1}((3t-2)\ell-(t-1)e)\big)&=h^1\big(\bF_1,\cO_{\bF_1}((3t-1)\ell-te)\big)=\\
&=h^2\big(\bF_1,\cO_{\bF_1}((3t-2)\ell-(t-1)e)\big)=0.
\end{align*}
Taking into account that $\omega_{\bF_1}\cong\cO_{\bF_1}(-3\ell+e)$, equality \eqref{Serre1} yields trivially
$$
h^2\big(\bF_1,\cO_{\bF_1}((3t-2)\ell-(t-1)e)\big)=h^0\big(\bF_1,\cO_{\bF_1}((-3t-1)\ell+te)\big)=0.
$$
The other vanishings follow from the ampleness of $\cO_{\bF_1}(3\ell-e)$ and the Nakai criterion. We conclude that $\cE$ is aCM.

We claim that $\cE$ is $\mu$--stable. To prove the claim recall that $\cO_{F}(\xi)\cong\varphi^*\cO_Q(\sigma_1)$, $\cO_F(\ell-e)\cong\varphi^*\cO_Q(\sigma_2)$. Thus, each fibre of the map $\varphi\colon F\to Q$ is a  curve in the class $(\ell-e)\xi$, hence a curve $B\in\Gamma_B$. The restriction of sequence \eqref{seqMinimalEarnest1} to $B\cong\p1$ yields $\cE\otimes\cO_B\cong\cO_{\p1}^{\oplus2}$. Since the morphism $\varphi$ is flat, it follows that $R^i\varphi_*\cE(\ell)=0$ for $i\ge0$ by \cite[Corollaries III.11.2 and III.12.9]{Ha2}. 

Pushing forward sequence \eqref{seqMinimalEarnest1} tensored by $\cO_F(\ell)$ we obtain the exact sequence
\begin{equation}
\label{seqMinimalQ}
0\longrightarrow \cA\longrightarrow \cO_{Q}(-\sigma_1)^{\oplus2}\oplus\cO_Q(-\sigma_2)\mapright{M}\cO_{Q}\longrightarrow 0,
\end{equation}
where $\cA:=\varphi_*\cE(\ell)$. We claim that $\cA$ is $\mu$--stable with respect to $\cO_Q(\sigma_1+\sigma_2)$. Indeed, the projection formula implies $m_1,m_2\in H^0\big(Q,\cO_Q(\sigma_1)\big)$, $m\in H^0\big(Q,\cO_Q(\sigma_2)\big)$. Thanks to Lemma \ref{lHoppe} we know that $\cA$ is $\mu$--stable with respect to $\cO_Q(\sigma_1+\sigma_2)$ if and only if  
\begin{equation}
\label{StableA}
a+b\ge-3/2\qquad \Rightarrow\qquad h^0\big(Q,\cA(-a\sigma_1-b\sigma_2)\big)=0.
\end{equation} 

Let $id_Q(a,b)$ be the identity of $\cO_Q(-a\sigma_1-b\sigma_2)$. The cohomology of sequence \eqref{seqMinimalQ} tensored by $\cO_Q(-a\sigma_1-b\sigma_2)$ implies that $H^0\big(Q,\cA(-a\sigma_1-b\sigma_2)\big)=0$ is the kernel of the map $M_Q(a,b)$ induced  in cohomology by $M\otimes id_Q(a,b)$. The same kind of argument used in the proof of Theorem \ref{tConstruction} yields
$$
h^0\big(Q,\cO_Q(-a\sigma_1-(b+1)\sigma_2)\big)=h^0\big(Q,\cO_Q(-(a+1)\sigma_1-b\sigma_2)\big)=0
$$
unless when $(a,b)$ is either $(0,-1)$ or $(-1,0)$. Since it is easy to check that both $M_Q(0,-1)$ and $M_Q(-1,0)$ are injective thanks to the restrictions on the entries $m_i$ and $m$, it follows that $\cA$ is $\mu$--stable.

We now prove the part of the statement concerning $\cS\cI_F(4\ell\xi-2e\xi+2 \ell^2)$. Since we already checked above the vanishing $R^i\varphi_*\cE(\ell)=0$ for $i\ge0$, it follows from  \cite[Exercise III.8.2]{Ha2} that $h^i\big(F,\cE\otimes\cE^\vee\big)=h^i\big(Q,\cA\otimes\cA^\vee\big)$ for each $i\ge0$. Thus $\cE$ is simple and $h^2\big(F,\cE\otimes\cE^\vee\big)=h^3\big(F,\cE\otimes\cE^\vee\big)=0$, thanks to Lemma \ref{lExt3}. Equality \eqref{Ext12} then returns $h^1\big(F,\cE\otimes\cE^\vee\big)=1$, hence $\cE$ corresponds to a smooth point of a unique $1$--dimensional component in $\cS\cI_F(4\ell\xi-2e\xi+2 \ell^2)$. 

It follows that all the components of $\cS\cI_F(4\ell\xi-2e\xi+2 \ell^2)$ are smooth and $1$--dimensional. The sheaf $\sHom_F\big(\cO_{F}(-\ell-\xi)^{\oplus2}\oplus\cO_F(-2\ell+e),\cO_{F}(-\ell)\big)$ is globally generated, hence $\cS\cI_F(4\ell\xi-2e\xi+2 \ell^2)$ is dominated by an open and non--empty subset of the affine space $\Hom_F\big(\cO_{F}(-\ell-\xi)^{\oplus2}\oplus\cO_F(-2\ell+e),\cO_{F}(-\ell)\big)$. Thus $\cS\cI_F(4\ell\xi-2e\xi+2 \ell^2)$ is irreducible and unirational, hence it is rational by the L\"uroth theorem.

Conversely, we will prove in the last part of the proof that if $\cA$ is a rank $2$ bundle on $Q$ with $c_1(\cA)=-2\sigma_1-\sigma_2$, $c_2(\cA)=2$ and $\mu$--stable with respect to $\cO_Q(\sigma_1+\sigma_2)$, then $\cE:=\varphi^*\cA(-\ell)$ is a $\mu$--stable instanton bundle with $c_2(\cE)=4\ell\xi-2 e\xi+2 \ell^2$. If this is true, then each instanton bundle $\cE$ with $c_2(\cE)=4\ell\xi-2 e\xi+2 \ell^2$ is necessarily $\mu$--stable, because $\cE\cong\varphi^*(\varphi_*\cE(\ell))(-\ell)$ thanks to the first part of the proof.

We claim that $h^0\big(F,\cE\big)=h^1\big(F,\cE\big)=0$. To this purpose we first check that $\cA^\vee\cong\cA(2\sigma_1+\sigma_2)$ is regular in the sense of Castelnuovo--Mumford. Indeed, Lemma \ref{lHoppe} yields $h^0\big(Q,\cA^\vee(-\sigma_1-\sigma_2)\big)=0$. Similarly $h^2\big(Q,\cA^\vee(-2\sigma_1-2\sigma_2)\big)=0$ and $h^2\big(Q,\cA^\vee(-\sigma_1-\sigma_2)\big)=0$, thanks to equality \eqref{Serre}. Equality \eqref{RRsurface} yields $h^1\big(Q,\cA^\vee(-\sigma_1-\sigma_2)\big)=0$. 

Thus the regularity of $\cA^\vee$ is proved, hence it is globally generated (see \cite{Mu}). It follows the existence of an exact sequence 
\begin{equation}
\label{seqAQ}
0\longrightarrow\cO_Q(-2\sigma_1-\sigma_2)\longrightarrow \cA\longrightarrow\cI_Y\longrightarrow0,
\end{equation}
where $Y$ is a set of $2$ distinct points. The pull--back of this exact sequence induces on $F$ the exact sequence
$$
0\longrightarrow\cO_F(-2\ell+e-2\xi)\longrightarrow \cE\longrightarrow\cI_Z(-\ell)\longrightarrow0,
$$
where $\cE(\ell):=\varphi^*\cA$ and $Z=\varphi^{-1}(Y)=B_1\cup B_2$ for disjoint curves $B_i\in\Gamma_B$. We deduce that  $c_1(\cE)=-h$ and $c_2(\cE)=4\ell\xi-2e\xi+2 \ell^2$ and $h^0\big(F,\cE\big)=0$.  

Equalities \eqref{Kuenneth} imply $h^1\big(F,\cO_F(-2\ell+e-2\xi)\big)=0$. Moreover, $\ell B_i=-1$ and $B_i\cong\p1$, hence $h^0\big(B_i,\cO_{B_i}(-\ell)\big)=0$: thus the cohomology of the sequence
$$
0\longrightarrow\cI_Z(-\ell)\longrightarrow \cO_F(-\ell)\longrightarrow\cO_Z(-\ell)\longrightarrow0
$$
returns $h^1\big(F,\cI_Z(-\ell)\big)=0$, hence $h^1\big(F,\cE\big)=0$. Thus the claim on the cohomology of $\cE$ is completely proved.

We claim that $\cE$ is $\mu$--stable. To prove the claim we will make use of Lemma \ref{lHoppe} and inequality \eqref{Hoppe}. It is not difficult to check that $h^0\big(F,\cO_F(-(a+2)\ell+(b+1)e-(c+2)\xi)\big)=0$ using equality \eqref{Kuenneth} when $3a+2c+6\ge b$. Thus the cohomology of sequence \eqref{seqAQ} tensored by $\cO_F(-a\ell+be-c\xi)$ yields 
\begin{align*}
h^0\big(F,\cE(-a\ell+be-c\xi)\big)&\le  h^0\big(F,\cI_Z(-(a+1)\ell+be-c\xi)\big)\le \\
&\le  h^0\big(F,\cO_F(-(a+1)\ell+be-c\xi)\big)
\end{align*}
if  $3a+2c+6\ge b$. One easily checks that the dimension on the right is zero, when $(a,b,c)$ is not one of the following: $(-1,0,0)$, $(-1,1,0)$, $(-1,2,0)$, $(-2,-1,0)$, $(-1,3,0)$, $(-2,0,0)$, $(-1,0,-1)$, $(-1,1,-1)$.

Since the curves in $\Gamma_B$ are not contained in $E$ (see Example \ref{eB}), it follows that we can exclude the cases $(-1,0,0)$, $(-1,1,0)$, $(-1,2,0)$, $(-1,3,0)$. Since $Z$ has two distinct components then also the case $(-2,-1,0)$ cannot occur. Thus we have only to deal with 
$$
h^0\big(F,\cE(2\ell)\big),\qquad h^0\big(F,\cE(\ell+\xi)\big),\qquad h^0\big(F,\cE(\ell+e+\xi)\big).
$$
Thanks to \cite[Exercises III.8.1 and III.8.3]{Ha2} and the isomorphisms $\varphi_*\cE(\ell)\cong\cA$ and $\varphi_*\cO_F(e)\cong\cO_Q\oplus\cO_Q(-\ell+e)$ we deduce
\begin{gather*}
h^0\big(F,\cE(2\ell)\big)=h^0\big(Q,\cA(\sigma_2)\oplus\cA\big),\qquad h^0\big(F,\cE(\ell+\xi)\big)=h^0\big(Q,\cA(\sigma_1)\big),\\
h^0\big(F,\cE(\ell+e+\xi)\big)=h^0\big(Q,\cA(\sigma_1)\oplus\cA(\sigma_1-\sigma_2)\big).
\end{gather*}
Their vanishing follows from the $\mu$--stability of $\cA$, thanks to implication \eqref{StableA}. We conclude that $\cE$ is also $\mu$--stable. 

Thanks to the previous computations, it then follows that it is  an instanton bundle as claimed. The argument for instanton bundles $\cE$ with $c_2(\cE)=3\ell\xi-e\xi+3 \ell^2$ is completely analogous.
\qed
\medbreak

Let $X$ be a threefold endowed with a very ample line bundle $\cO_X(h)$. Recall that a vector bundle $\cF$ on $X$ is called {\sl weakly Ulrich (with respect to $\cO_X(h)$)} if $h^i\big(X,\cF(th)\big)=0$ if $i=0$ and $t\le -2$, or $i=1$ and $t\ne -1,-2$, or $i=2$ and $t\ne -2,-3$, or $i=3$ and $t\ge-2$. The vector bundle $\cF$ is called {\sl Ulrich (with respect to $\cO_X(h)$)} if it is aCM and $h^0\big(X,\cF(-h)\big)=h^3\big(X,\cF(-3h)\big)=0$. See \cite{E--S--W} for the general definition and properties of weakly Ulrich and Ulrich bundles: in particular it is proved therein that each Ulrich bundle is weakly Ulrich.

The link between minimal instanton bundles on a Fano threefold $X$ with $i_X\ge3$ and weakly Ulrich bundles is given by the following result.

\begin{proposition}
Let $X$ be a Fano threefold with $i_X\ge3$.

A rank $2$ vector bundle $\cE$ on $X$ is a minimal instanton bundle if and only if $\cF:=\cE((2-q_X)h)$ is weakly Ulrich with
\begin{equation}
\label{Chern}
c_1(\cF)=(4-i_X)h,\qquad c_2(\cF)h=k_X+(2+q_X-i_X)(2-q_X)\deg(X).
\end{equation}
\end{proposition}
\begin{proof}
It is trivial to check that equalities \eqref{Chern} hold for $\cF$ if and only if $c_1(\cE)=(2q_X-i_X)h$ and $c_2(\cE)h=k_X$.

Let $i_X=4$, so that $2=q_X$, i.e. $X=\p3$ and $\cE=\cF$ is a minimal instanton bundle. Thus $\cE$ is $\mu$--stable, $c_1(\cE)=0$ and $c_2(\cE)h=1$. Thus \cite[Lemma II.4.3.2]{O--S--S} implies that  $\cE$ is a null--correlation bundle i.e. it fits into an exact sequence of the form
\begin{equation}
\label{seqNull}
0\longrightarrow\cO_{\p3}(-1)\longrightarrow\Omega_{\p3}^1(1)\longrightarrow\cE\longrightarrow0.
\end{equation}
We now show that each null--correlation bundle is weakly Ulrich. The cohomology of sequence \eqref{seqNull} returns $h^0\big(X,\cF(th)\big)=0$ for $t\le 0$, and $h^1\big(X,\cF(th)\big)=0$ for $t\ne-1$. Thanks to equality \eqref{Serre}, we then deduce $h^0\big(X,\cF(th)\big)=0$ for $t\ge -4$, and $h^2\big(X,\cF(th)\big)=0$ for $t\ne-3$. We deduce that $\cF$ is weakly Ulrich.

Conversely, assume that $\cF$ is a rank $2$ weakly Ulrich bundle on $X=\p3$ with $c_1(\cF)=0$ and $c_2(\cF)h=1$. Equality \eqref{RRgeneral} implies $\chi(\cF)=0$. Since $h^i\big(X,\cF\big)=0$ for $i\ne0$, it follows that $h^0\big(X,\cF\big)=0$ as well. Thus $\cF$ is $\mu$--stable thanks to Lemma \ref{lHoppe}, hence it is a null--correlation bundle thanks to \cite[Lemma II.4.3.2]{O--S--S}.

Let $i_X=3$, so that $1=q_X$, i.e. $X=Q$ is a smooth quadric and $\cE=\cF(-h)$ is a minimal instanton bundle. Thus \cite[Remark 4.3]{C--C--G--M} implies that $\cF$ is Ulrich. 

Conversely let $\cF$ be  a rank $2$ weakly Ulrich bundle on $X=Q$ with $c_1(\cF)=h$ and $c_2(\cF)h=1$. By definition $h^0\big(X,\cF(-2h)\big)=0$. Let $h^0\big(X,\cF(-h)\big)\ne0$: since $\Pic(X)$ is cyclic and generated by $\cO_X(h)$ it follows that each non--zero section $s\in H^0\big(X,\cF(-h)\big)$ vanishes exactly along a curve $Z$ and
$$
\det(\cN_{Z\vert X})\cong\det(\cF(-h))\otimes\cO_Z\cong\cO_X(-h)\otimes\cO_Z.
$$
The degree of $Z$ is $c_2(\cF(-h))h=1$, hence $Z$ is a line. The adjunction formula on $X$ then returns 
$$
\det(\cN_{Z\vert X})\cong\omega_Z\otimes\cO_X(3h)\cong\cO_X(h)\otimes\cO_Z.
$$
The contradiction above implies that $h^0\big(X,\cF(-h)\big)=0$, hence $\cF(-h)$ is $\mu$--stable thanks to Lemma \ref{lHoppe}. 

By definition we know that $h^1\big(X,\cF(-th)\big)=0$ for $t\ne-2,-1$. Equality \eqref{RRgeneral} gives $h^1\big(X,\cF(-h)\big)=0$. Let $H\subseteq X$ be a hyperplane and consider the exact sequence
\begin{equation}
\label{seqRestriction}
0\longrightarrow\cO_X(-h)\longrightarrow\cO_X\longrightarrow\cO_H\longrightarrow0.
\end{equation}
Since $\cF$ is $\mu$--semistable, then the same is true for $\cF\otimes\cO_H$, thanks to \cite[Theorem 3]{Ma}. Thus the cohomology of sequence \eqref{seqRestriction} tensored by $\cF(-h)$ and $h^1\big(X,\cF(-h)\big)=0$ imply that $h^1\big(X,\cF(-2h)\big)=0$. We have $h^i\big(X,\cF(th)\big)=0$ for $i=1,2$ and $t\in\bZ$, thanks to equality \eqref{Serre}. We easily deduce from the Chern classes of $\cF$ that it cannot split as a sum of two line bundles, hence it necessarily coincides with $\cS(h)$, where $\cS$ is the spinor bundle on $X$ (see \cite{Ott1}).
\end{proof}

\begin{remark}
In the above proof we have checked that when $i_X=3$, then Ulrich bundles of rank $2$ are exactly the weakly Ulrich ones.
\end{remark}

It is natural to ask if a similar result holds when $i_X\le2$. In \cite[Remark 4.3]{C--C--G--M} we also showed that when $i_X=2$ (hence $q_X=1$), $\cO_X(h)$ is very ample and $\cE$ is a minimal instanton bundle, then $\cF=\cE((2-q_X)h)\cong\cE(h)$ is Ulrich, hence weakly Ulrich: in this case $\cF$ still satisfies equalities \eqref{Chern}. 

Thus it is natural to ask if weakly Ulrich bundles are automatically instanton bundles up to twists when $i_X\le2$. This is not the case as the following example and remark show. 

\begin{example}
Let $X$ be a Fano threefold with $i_X=2$ and  take the union $Z\subseteq X$ of two disjoint lines. The scheme $Z$ is a locally complete intersection subscheme of degree $2$ and genus $p_a(Z)=-1$. Moreover, $\omega_Z\cong\cO_Z(-2h)$, hence the adjunction formula on $X$ implies $\det(\cN_{Z\vert X})\cong\cO_Z$. Since $h^i\big(X,\cO_X\big)=0$ for $i\ge1$, it follows from Theorem \ref{tSerre} the existence of an exact sequence of the form
$$
0\longrightarrow \cO_X\longrightarrow \cE\longrightarrow \cI_{Z\vert X}\longrightarrow 0.
$$
The cohomology of the above sequence tensored by $\cO_X(th)$ yields $h^0\big(X,\cE(th)\big)=0$ if $t\le-1$. Similarly, the cohomology of sequence \eqref{seqIdeal} tensored by $\cO_X(th)$ implies $h^1\big(X,\cI_{Z\vert X}(th)\big)=0$ if $t\ne0$. It follows that $h^1\big(X,\cE(th)\big)=0$ in the same range. Equality \eqref{Serre1} finally yields that $\cF:=\cE(h)$ is weakly Ulrich and satisfies equalities \eqref{Chern}.

On the other hand $c_1(\cE)=0$, $c_2(\cE)=2$ and $h^1\big(X,\cE(-h)\big)=0$. Moreover, it is not difficult to check the $\mu$--semistability of $\cE$ using the classification of Fano threefold with $i_X=2$ (see \cite{I--P}) and Lemma \ref{lHoppe}. Nevertheless $\cE$ is obviously not an instanton bundle because $h^0\big(X,\cE\big)=1$.
\end{example}

\begin{remark}
When $i_X=1$ (hence $q_X=0$) the picture is even vaguer, because the number $k_X$ is generally not attained by instanton bundles on $X$ (see \cite{Fa} where it is proved a sharp lower bound for $c_2(\cE)h$ when $\varrho_X=1$). Anyhow, in all the examples in the literature, every minimal instanton bundle $\cE$ is aCM (e.g. see \cite{Fa, Cs--Ge} and Theorem \ref{tMinimal} above). If $\cF:=\cE((2-q_X)h)\cong\cE(2h)$, then $h^0\big(X,\cF(-2h)\big)=h^3\big(X,\cF(-2h)\big)=0$. It follows that $\cF$ is still weakly Ulrich in all these cases.
\end{remark}

Thus it is natural to ask the following questions.

\begin{question}
Is each minimal instanton bundle on $X$ weakly Ulrich up to twists when $i_X=1$? 
\end{question}

\begin{question}
Is it possible to classify weakly Ulrich bundles $\cF$ of rank $2$ on $X$ with $c_1(\cF)=(4-i_X)h$, at least for low values of  $c_2(\cF)h$?
\end{question}

\bigskip
\noindent
Vincenzo Antonelli,\\
Dipartimento di Scienze Matematiche, Politecnico di Torino,\\
c.so Duca degli Abruzzi 24,\\
10129 Torino, Italy\\
e-mail: {\tt vincenzo.antonelli@polito.it}

\bigskip
\noindent
Gianfranco Casnati,\\
Dipartimento di Scienze Matematiche, Politecnico di Torino,\\
c.so Duca degli Abruzzi 24,\\
10129 Torino, Italy\\
e-mail: {\tt gianfranco.casnati@polito.it}

\bigskip
\noindent
Ozhan Genc,\\
Faculty of Mathematics and Computer Science, Jagiellonian University,\\
ul. {\L}ojasiewicza 6,\\
30-348 Krak{\'o}w, Poland\\
e-mail: {\tt ozhangenc@gmail.com}

\end{document}